\documentclass[a4paper,12pt]{article}

\title{\Large New global Carleman estimates and null controllability for a stochastic Cahn-Hilliard type equation}
\author{\sc\normalsize Sen Zhang$^\dagger$, Hang Gao$^\S$, and  Ganghua Yuan$^\ddagger$}
\date{}

\usepackage{amsmath}
\usepackage{amsthm}
\usepackage{amsfonts,amssymb}
\usepackage{geometry}
\usepackage{mathrsfs}
\usepackage{indentfirst}
\usepackage{titlesec}
\usepackage{color}

\newtheorem{theorem}{Theorem}[section]

\newtheorem{lemma}{Lemma}[section]

\newtheorem{remark}{Remark}[section]

\numberwithin{equation}{section}

\newcommand{\dd}[0]{\mathrm{d}}
\newcommand{\EE}[0]{\mathbb{E}}
\newcommand{\dxt}[0]{\dd x\dd t}

\newcommand{\dt}[0]{\dd t}
\newcommand{\dx}[0]{\dd x}
\newcommand{\intt}[0]{\int_0^T}

\newcommand{\barr}[1]{{\overline {#1}}}

\titleformat*{\section}{\normalsize\bfseries\rmfamily}
\titleformat*{\subsection}{\normalsize\bfseries\rmfamily}

\begin{document}
\maketitle

\begin{abstract}
In this paper, we study the null controllability for a stochastic semilinear Cahn-Hilliard type equation, whose semilinear term contains first and second order derivatives of solutions. To start with, an improved global Carleman estimate for linear backward stochastic fourth order parabolic equations with $L^2$-valued source terms is derived, which is based on a new fundamental identity for a stochastic fourth order parabolic operator. Based on it, we establish a new global Carleman estimate for linear backward stochastic fourth order parabolic equations with $H^{-2}$-valued source terms, which, together with a fixed point argument, derive the desired null controllability for the stochastic Cahn-Hilliard type equation.
\end{abstract}
\footnote{
{\it Key  Words: null controllability,  stochastic Cahn-Hilliard equation,   Carleman estimate}


\quad $^\dagger$School of Mathematics and Statistics, Beijing Institute of Technology, Beijing 100081, P.R. China (7520230216@bit.edu.cn).

\quad $^\S$KLAS, School of Mathematics \& Statistics,  Northeast Normal University, Changchun Jilin 130024, P.R. China (hangg@nenu.edu.cn)

\quad $^\ddagger$KLAS, School of Mathematics \& Statistics,  Northeast Normal University, Changchun Jilin 130024, P.R. China (yuangh925@nenu.edu.cn).
}

\section{Introduction and Main Results}
\subsection{Motivation}
Cahn-Hilliard equation is an important nonlinear mathematical physics equation, which describes the process of phase separation, by which the two components of a binary fluid spontaneously separate and form domains pure in each component. It was proposed initially by Cahn \cite{Cahn1961On, Cahn1958Free} and Hilliard \cite{Hilliard1958Spinodal} in order to describe isothermal phase separation processes in binary alloys.  In the past years, the Cahn-Hilliard equation and its variants not only play an important role in materials science \cite{Erlebacher1997Long}, but also have been successfully applied in different other fields such as population dynamics \cite{Cohen1981A}, thin films theory \cite{Oron1997Long}, chemistry \cite{Verdasca1995Chemically}, image processing \cite{Bertozzi2007Inpainting} and tumor growth simulation \cite{Hilhorst2015Formal}.   For more physical background and applications, we refer to \cite{Kim2016Basic, Miranville2019The} and the references therein.
We can also refer to \cite{Elliott1986On,Novick1989Non,Novick1990On} for general analysis and discussion of the deterministic Cahn-Hilliard equation.

In practice, the deterministic Cahn-Hilliard equation presents some drawbacks. In fact, due to unpredictable movements at the atomistic level, which may be caused, for example, by temperature oscillations, magnetic effects, or configurational interactions, the phase-separation process inevitably presents some disruptions, acting at a microscopic level. Therefore, the most natural way to overcome this problem is to replace deterministic functions with stochastic processes as mathematical descriptions, and the resulting model becomes a stochastic Cahn-Hilliard equation. This was proposed by Cook for Wiener noise and produced the well-known Cahn-Hilliard-Cook stochastic model \cite{Cook1970Brownian}.
Subsequently, the stochastic version of the model has been repeatedly confirmed to be the only one that can truly describe the phase-separation in the alloy \cite{Binder1981Kinetics,Pego1989Front}.
Since then, stochastic Cahn-Hilliard equations have been extensively studied, such as the well-posedness problems (see e.g., \cite{Antonopoulou2016Existence,Cardon-Weber2001Cahn-Hilliard} and references therein) and numerical simulation problems (see e.g., \cite{Cui2021Strong,Furihata2018Strong} and references therein).
Recently, many efforts have been devoted to studying the control problems of stochastic Cahn-Hilliard type equations and we refer to \cite{Scarpa2019Optimal} for some known results.
However, as far as we know, little is known about the controllability of stochastic Cahn-Hilliard equations. Motivated by this, we focus on this point in this paper.
\subsection{Statement of main results}
First, we give some notations and basic assumptions. Let $T>0$, $G:=(0,1)$, $Q:=(0,T)\times G$. Assume $G_{0}$ to be a given nonempty open subset of $G$ and denote by $\chi_{G_{0}}$ the characteristic function of the set $G_{0}$. We use $C$ to denote a generic positive constant depending only on $G_{0}$ and $G$,  which may change from line to line.

Let $(\Omega,\mathcal{F},\mathbf{F},\mathbb{P})$ with $\mathbf{F}=\{\mathcal{F}_t\}_{t\geq0}$ be a complete filtered probability space on which a one-dimensional standard Brownian motion $\{B(t)\}_{t\geq0}$ is defined and $\mathbf{F}$ is the natural filtration generated by $B(\cdot)$, augmented by all the $\mathbb{P}$ null sets in $\mathcal{F}$. Let $H$ be a Banach space.
Denote by $L^2_{\mathcal{F}_t}(\Omega;H)$ the space of all $\mathcal{F}_t$-measurable random variables $\zeta$ such that $\EE |\zeta|^2_{H}<\infty$; denote by $L^2_\mathbb{F}(0,T;H)$ the space consisting of all $H$-valued $\mathbf{F}$-adapted processes $X(\cdot)$ such that $\mathbb{E}(|X(\cdot)|^2_{L^2(0,T;H)})<\infty$; by $L^\infty_\mathbb{F}(0,T;H)$ the space consisting of all $H$-valued $\mathbf{F}$-adapted bounded processes $X(\cdot)$; and by $L^2_\mathbb{F}(\Omega;C([0,T];H))$ the space consisting of all $H$-valued $\mathbf{F}$-adapted continuous processes $X(\cdot)$ such that $\mathbb{E}(|X(\cdot)|^2_{C(0,T;H)})<\infty$. Similarly, one can define $L^2_\mathbb{F}(\Omega;C^k([0,T];H))$ for any positive integer $k$. All of these spaces are Banach spaces with canonical norms.

This paper is concerned with the following stochastic Cahn-Hilliard type equation
\begin{equation}\label{equationy1}
	\left\{
		\begin{aligned}
	&\dd y+y_{xxxx}\dt=[f(\omega,t,x,y,y_{x},y_{xx})+\chi_{G_0}u]\dt\\
	&\quad\quad\quad\quad\quad\quad\quad\quad\quad\quad\quad +[g(\omega,t,x,y,y_{x},y_{xx})+U]\dd B(t) &\textup{in}\ &Q,\\
    &y(t,0)=y(t,1)=0 &\textup{in}\ &(0,T),\\
    &y_{x}(t,0)=y_{x}(t,1)=0 &\textup{in}\ &(0,T),\\
    &y(0)=y_0 &\textup{in}\ &G,
            \end{aligned}
    \right.
\end{equation}
where $(u,U)$ is the control variable, $y$ is the state variable, and $y_{0}$ is a given initial value.
For simplicity, we use the notation $y_{x}:=\frac{\partial y}{\partial x}$. In a similar manner, we use the notation $w_{x}$, $r_{x}$, $h_{x}$, etc. for the partial derivatives of $w$, $r$, and $h$ with respect to $x$.

We assume that the functions $f$ and $g$ satisfy the following assumptions:
\begin{equation}\nonumber
\begin{split}
{\bf (A_{1}).}\ & \mbox{For each}\ y\in H_{0}^2(G),\ f(\cdot,\cdot,\cdot,y,y_{x},y_{xx})\ \mbox{and}\ g(\cdot,\cdot,\cdot,y,y_{x},y_{xx})\  \mbox{are}\ \mathbf{F}\mbox{-adapted and}\
\\
& L^2\mbox{-valued stochastic processes}.
\\
{\bf (A_{2}).}\ &\forall\ (\omega,t,x)\in \Omega\times Q,\ f(\omega,t,x,0,0,0)=0.
\\
{\bf (A_{3}).}\ &\exists \ \kappa>0,\ \forall\ (\omega,t,x,a_{1},b_{1},s_{1},a_{2},b_{2},s_{2})\in \Omega\times Q\times \mathbb{R}^6,
\\
&|f(\omega,t,x,a_{1},b_{1},s_{1})-f(\omega,t,x,a_{2},b_{2},s_{2})|\leq \kappa (|a_{1}-a_{2}|+|b_{1}-b_{2}|+|s_{1}-s_{2}|).
\\
{\bf (A_{4}).}\ &\exists \ \kappa_{1}>0,\ \forall\ (\omega,t,x,a_{1},b_{1},s_{1},a_{2},b_{2},s_{2})\in \Omega\times Q\times \mathbb{R}^6,
\\
&|g(\omega,t,x,a_{1},b_{1},s_{1})-g(\omega,t,x,a_{2},b_{2},s_{2})|\leq \kappa_{1} (|a_{1}-a_{2}|+|b_{1}-b_{2}|+|s_{1}-s_{2}|).
\end{split}
\end{equation}

\begin{remark}
Under the above assumptions, by taking $y_{0}\in L^2_{\mathcal{F}_{0}}(\Omega;L^{2}(G))$ and $(u,U)\in L^2_{\mathbb{F}}(0,T;L^2(G_{0}))\times L^2_{\mathbb{F}}(0,T;L^2(G))$, one can easily show that (\ref{equationy1}) admits a unique solution
\begin{equation}\nonumber
y\in \mathcal{H}_{T}	
:=
L^2_{\mathbb{F}}(\Omega; C([0,T];L^2(G))) \cap L^2_{\mathbb{F}}(0,T;H_0^2(G))
\end{equation}
by the standard fixed point argument (see e.g., \cite{Lu2021Mathematical}). We omit the proof here.
\end{remark}

The main purpose of this paper is to study the null controllability of the system (\ref{equationy1}).
System (\ref{equationy1}) is said to be {\it globally null controllable} if for any initial datum $y_{0}\in L^2_{\mathcal{F}_{0}}(\Omega;L^2(G))$, there exists a control $(u,U)\in L^2_{\mathbb{F}}(0,T;L^2(G_{0}))\times L^2_{\mathbb{F}}(0,T;L^2(G))$ such that the corresponding solution $y$ of (\ref{equationy1})  satisfies $y(T,\cdot)=0$ in $G$, $\mathbb{P}$-a.s.

The controllability of deterministic nonlinear partial differential equations (PDEs) has been studied by many authors and the results are relatively rich, such as
nonlinear parabolic equation (see e.g.,  \cite{Barbu2000Exact,Doubova2002On,Emanuilov1995Controllability,Fabre1995Approximate,Fernandez-Cara1995Null,Fernandez-Cara2000Null,Fursikov1996Controllability,LeBalch2020Global}),
nonlinear fourth order  parabolic equation (see e.g., \cite{Kassab2020Null,Zhou2012Observability}),
Ginzburg-Landau equation (see e.g., \cite{Fu2009Null,Rosier2009Null}),
Kuramoto-Sivashinsky equation (see e.g., \cite{Cerpa2010Null,Cerpa2011Local}) and
Cahn-Hilliard equation (see e.g., \cite{Gao2016A}).
It can be seen from these results that the author usually uses the following strategies to study controllability problems.
First,  linearize the nonlinear system and study the controllability of the linearized system.
Then the controllability problem of nonlinear systems is solved by using appropriate fixed point methods, usually Schauder or Kakutani fixed point methods.
At this point, the property of compactness plays a crucial role, for example, the compact embedding result of the Aubin-Lions lemma is used commonly.

In recent years, the controllability of stochastic PDEs (SPDEs) has received much attention.
However, being compared with the results for deterministic PDEs, little has been known in the stochastic setting.
In this respect, we refer to \cite{Barbu2003Carleman,Fu2017A,Fu2017Controllability,Gao2015Observability,Liao2024Exact,Lu2011Some,Lu2013Exact,Lu2022Null,Tang2009Null,Yu2022Carleman} for some known results.
In \cite{Barbu2003Carleman,Lu2011Some,Tang2009Null}, the null controllability and approximate controllability of the stochastic parabolic equations were obtained.
Reference \cite{Fu2017A} and \cite{Fu2017Controllability} showed the null controllability of the stochastic complex  Ginzburg-Landau equation.
In \cite{Lu2013Exact}, the exact controllability for stochastic Schr\"odinger equation was established.
In \cite{Liao2024Exact} and \cite{Yu2022Carleman}, the exact controllability of a refined stochastic wave equation and a refined stochastic beam equation were discussed respectively.
References \cite{Gao2015Observability} and \cite{Lu2022Null} showed the null controllability of the stochastic fourth order parabolic type equations.
It should be noted that the above works mainly focus on the controllability of linear SPDEs systems, while there is little literature on the controllability of nonlinear stochastic systems.
This is due to the lack of compactness embedding for the function spaces related to SPDEs (see \cite[Remark 2.5]{Tang2009Null} or \cite{Lu2021Mathematical}), which makes some classical strategies in deterministic setting (see e.g., \cite{Fernandez-Cara2000Null}) fail to establish null controllability for semilinear systems at the stochastic level.
To our best knowledge, the known results in this direction seem to be only \cite{Hernandez-Santamaria2022Statistical,Hernandez-Santamaria2023Global,Zhang2024New}, in which the author established the null controllability of semilinear stochastic parabolic equations.
However, as far as we know, nothing is known about the controllability of nonlinear stochastic high-order parabolic equations, especially for the stochastic Cahn-Hilliard equation with the nonlinear term including state and  derivative of the state,
This kind of nonlinear term will make more difficult to establish the controllability.

To overcome the lack of compactness mentioned above, we borrow some ideas from \cite{Hernandez-Santamaria2023Global,Liu2014Global} when proving the null controllability of (\ref{equationy1}).  First, we obtain a new global Carleman estimate for a backward stochastic fourth order parabolic equation with a $L^2$-valued source by introducing a suitable singular weighted function.
Next, by combining the $L^2$-Carleman estimate with the duality argument and Hilbert Uniqueness Method (HUM) introduced in \cite{Lions1988Exact}, we get a new global Carleman estimate for backward stochastic fourth order parabolic equation with a $H^{-2}$-valued source.
Using the above $H^{-2}$-valued Carleman estimate, we will establish the null controllability for a linear system with a source $\phi$ of the form
\begin{equation}\nonumber
	\left\{
		\begin{aligned}
	&\dd y+y_{xxxx}\dt=(\phi+\chi_{G_0}u)\dt+U\dd B(t) &\textup{in}\ &Q,\\
    &y(t,0)=y(t,1)=0 &\textup{in}\ &(0,T),\\
    &y_{x}(t,0)=y_{x}(t,1)=0 &\textup{in}\ &(0,T),\\
    &y(0)=y_0 &\textup{in}\ &G,
            \end{aligned}
    \right.
\end{equation}
where the source $\phi$ is in some suitable functional space.
At last, we need to prove that a nonlinear mapping $\phi\rightarrow f(\omega,t,x,y,y_{x},y_{xx})$ is strictly contractive in a suitable functional space. Through a Banach fixed point method which does not rely on any compactness argument, we can obtain the null controllability for (\ref{equationy1}).
The main novelty here is that we construct a new weight function in proving the Carleman estimates and use a Banach fixed point method to prove the null controllability for the system (\ref{equationy1}).

The main controllability result in this paper is as follows.
\begin{theorem}\label{control}
	Let assumptions $(A_{1})$, $(A_{2})$ and $(A_{3})$ be satisfied. Then system $(\ref{equationy1})$ is globally null controllable.
\end{theorem}
\begin{remark}\label{remark1}
We claim that the null controllability of equation (\ref{equationy1}) can be reduced to the null controllability of
\begin{equation}\label{equationy2}
	\left\{
		\begin{aligned}
	&\dd y+y_{xxxx}\dt=[f(\omega,t,x,y,y_{x},y_{xx})+\chi_{G_0}u]\dt+U\dd B(t) &\textup{in}\ &Q,\\
    &y(t,0)=y(t,1)=0 &\textup{in}\ &(0,T),\\
    &y_{x}(t,0)=y_{x}(t,1)=0 &\textup{in}\ &(0,T),\\
    &y(0)=y_0 &\textup{in}\ &G.
            \end{aligned}
    \right.	
\end{equation}
Indeed, assume that one can find two controls $u\in L^2_{\mathbb{F}}(0,T;L^2(G_{0}))$ and $U\in  L^2_{\mathbb{F}}(0,T;L^2(G))$ such that the corresponding solution $y$ of (\ref{equationy2}) satisfies $y(T,\cdot)=0$ in $G$, $\mathbb{P}$-a.s. Since the control $U$ is distributed in the whole domain $G$, the state $y$ still satisfies equation (\ref{equationy1}) with the controls $u^{*}=u$ and
\begin{equation}\nonumber
	U^{*}=U-g(\omega,t,x,y,y_{x},y_{xx})\in L^2_{\mathbb{F}}(0,T;L^2(G)),
\end{equation}
which is well defined by $(A_{1})$. Moreover, we still have the
controllability property i.e. $y(T,\cdot)=0$ in $G$, $\mathbb{P}$-a.s. This is why we can drop the Lipschitz condition $(A_{4})$ on $g$ in Theorem \ref{control}.
\end{remark}
\begin{remark}
From the proof of Theorem \ref{control}, it is not difficult to see that null controllability is still valid when the boundary condition in the system (\ref{equationy1}) is replaced by
\begin{equation}\label{102}
\left\{
		\begin{aligned}
    &y(t,0)=y(t,1)=0 &\textup{in}\ &(0,T),\\
    &y_{xx}(t,0)=y_{xx}(t,1)=0 &\textup{in}\ &(0,T),
         \end{aligned}
    \right.
\end{equation}
the system is still null controllability.
Indeed, we can get the  Carleman estimate for the backward system with boundary (\ref{102}), this is a key point for proving the null controllability (see Remark \ref{remark01} and Remark \ref{remark02} for an explanation).
\end{remark}
\begin{remark}
The controllability of the high-dimensional Cahn-Hilliard type equation may be able to obtained by the method in this paper, where the key point is to get the Carleman estimates similar in Theorem \ref{carle2} and Theorem \ref{carle1}. However, it is more difficult to derive the Carleman estimates in the high-dimensional case (see e.g., \cite{Lu2022Null,Zhang2024Unique}), and we will consider it in the future.
\end{remark}
\begin{remark}
Theorem \ref{control} requires an extra control $U\in  L^2_{\mathbb{F}}(0,T;L^2(G))$ on the diffusion term.
It would be quite interesting to establish the null controllability for (\ref{equationy1}) by only one control force or the control $U$ acting only on a sub-domain of $G$.
However, this seems difficult for us.
In fact, these problems are still open, even for general linear SPDEs (see e.g., \cite{Fu2017Controllability,Gao2015Observability,Tang2009Null}).
\end{remark}
As mentioned above, to prove the null controllability of (\ref{equationy2}), we need to establish the Carleman estimates for the following backward stochastic fourth order parabolic equation with a source in drift term:
\begin{equation}\label{equationr}
	\left\{
		\begin{aligned}
	&\dd r-r_{xxxx}\dt=\Phi\dt+R\dd B(t) &\textup{in}\ &Q,\\
    &r(t,0)=r(t,1)=0 &\textup{in}\ &(0,T),\\
    &r_{x}(t,0)=r_{x}(t,1)=0 &\textup{in}\ &(0,T),\\
    &r(T)=r_T &\textup{in}\ &G,
         \end{aligned}
    \right.
\end{equation}
where $(r,R)$ denotes the solution associated to the terminal data $r_{T}$ and we assume that $\Phi=\phi_0+\phi_{1x}+\phi_{2xx}$ with $\phi_{0},\phi_{1},\phi_{2}\in L^2_{\mathbb{F}}(0,T;L^2(G)) $.
Noting that $\Phi\in L^2_{\mathbb{F}}(0,T;H^{-2}(G))$, when $r_{T}\in L^2_{\mathcal{F}_{T}}(\Omega;L^{2}(G))$, one can show that (\ref{equationr})  admits a unique solution $(r,R)\in \mathcal{H}_{T}\times L^2_{\mathbb{F}}(0,T;L^2(G))$ (see \cite[Theorem 4.11]{Lu2021Mathematical}).

To state our Carleman estimates for (\ref{equationr}), we first introduce the weight function. To begin with, we introduce some auxiliary functions.
Let $G'$ be any given nonempty open subset of $G$ satisfying $G'\subset\subset G_{0}$. It can be seen from \cite{Fursikov1996Controllability}  that there exists a function $\beta\in C^4(\barr{G})$ such that
\begin{equation}\label{beta}
	\left\{
		\begin{aligned}
	&0<\beta(x)\leq 1\quad\forall x\in G,\\
	&\beta(x)=0\quad \forall x\in \{0,1\},\\
	&\inf_{G\backslash \barr{G'}}\{|\beta'(x)|\}\geq \alpha_0 >0.
	    \end{aligned}
    \right.
\end{equation}
Without loss of generality, in what follows we assume that $0<T<1$. For some constant $m\geq 1$ and $\sigma \geq 2$, we define the following weight function depending on the time variable:
\begin{equation}\label{gamma}
	\left\{
		\begin{aligned}
	&\gamma(t)=1+\bigg(1-\frac{4t}{T}\bigg)^{\sigma}, &t\in &[0,T/4),\\
	&\gamma(t)=1,&t\in &[T/4,T/2),\\
	&\gamma(t)\ \text{is increasing on}\ [T/2,3T/4),\\
	&\gamma(t)=\frac{1}{(T-t)^m},& t\in &[3T/4,T),\\
	&\gamma(t)\in C^2([0,T)).
	    \end{aligned}
    \right.
\end{equation}
Set
\begin{equation}\label{101}
    \alpha(x)=e^{\mu(10m+\beta(x))}-\mu e^{\mu(10m+10)},\quad
	\varphi(t,x)=\gamma(t)\alpha(x),\quad
    \xi(t,x)=\gamma(t)e^{\mu(10m+\beta(x))},
\end{equation}
where $\mu$ is a positive parameter with $\mu\geq 2$ and $\sigma$ is chosen as
\begin{equation}\label{sigma}
\sigma=\lambda^3\mu^4e^{\mu (30m-6)}	
\end{equation}
for some parameter $\lambda\geq 1$.
We finally set the weight function
\begin{equation}\label{theta}
	\theta =e^{\ell},\ \text{where}\ \ell(t,x)=\lambda\varphi(t,x).
\end{equation}
\begin{remark}
It can be seen from the (\ref{theta})  that, compared with the classical weight function (see, e.g. \cite{Tang2009Null}), the main difference here is that the weight does not degenerate as $t\rightarrow 0^+$.
From the choice of (\ref{101}) and (\ref{sigma}), it can be seen that our weight function is also different from that in \cite{Hernandez-Santamaria2023Global} and \cite{Zhang2024New}.
\end{remark}
\begin{remark}
Here, due to the assumement $0<T<1$, we actually obtain that the system (\ref{equationy1}) is globally null controllable at any small time.
\end{remark}
Now, we state the results of the Carleman estimate.
The first result provides a new global Carleman estimate for (\ref{equationr}) with the drift term taking values in $L^2(G)$.
\begin{theorem}\label{carle2}
Assume that $\phi_{0}\in L^2_{\mathbb{F}}(0,T;L^2(G))$, $\phi_{1}\equiv 0$,  $\phi_{2}\equiv 0$ in $(\ref{equationr})$, then there exist $\lambda_{0}>0$ and $\mu_{0}>0$ such that the unique solution $(r,R)\in \mathcal{H}_{T}\times L^2_{\mathbb{F}}(0,T;L^2(G))$ to $(\ref{equationr})$ with respect to $r_{T}\in L^2_{\mathcal{F}_{T}}(\Omega;L^{2}(G))$ satisfies
\begin{equation}\label{carest2}	
    \begin{split}
&\lambda^4\mu^5 e^{4\mu(10m+1)}\EE \|\theta(0,\cdot) r(0,\cdot)\|_{L^{2}(G)}^2
+\EE \int_{Q}\theta^2\lambda^3\mu^4\xi^3|r_{xx}|^2\dxt\\
&+\EE \int_{Q}\theta^2\lambda^5\mu^6\xi^5|r_{x}|^2\dxt
+\EE \int_{Q}\theta^2\lambda^7\mu^8\xi^7|r|^2\dxt \\
\leq\ &
C\bigg(\EE \int_{0}^{T}\int_{G_0} \theta^2\lambda^7\mu^8\xi^7|r|^2\dxt
+\EE \int_{Q}\theta^2|\phi_0|^2\dxt
+\EE \int_{Q}\theta^2\lambda^4\mu^4\xi^5|R|^2\dxt
\bigg),
\end{split}
\end{equation}
for all $\lambda\geq\lambda_{0}$ and $\mu\geq\mu_{0}$, where  $C>0$ only depends on $G$ and $G_0$.
\end{theorem}

\begin{remark}
Since the weight function $\gamma$ does not blow up as $t\rightarrow 0$, it prevents $\theta$ from vanishing at $t=0$.
This change brings additional difficulties compared with the classical proof of the Carleman estimate for the backward stochastic fourth order parabolic equation (see e.g., \cite[Theorem 4.1]{Gao2015Observability}),  and leads to the appearance of the first term on the left side of (\ref{carest2}).
\end{remark}
\begin{remark}
This type of Carleman estimate was first considered in \cite{Badra2016Local} to
deal with the local trajectory controllability for the incompressible Navier-Stokes equations.
Later, references \cite{Hernandez-Santamaria2023Global} and \cite{Zhang2024New} developed the ideas in \cite{Badra2016Local} to study the global null controllability of stochastic semilinear parabolic equations.
\end{remark}
\begin{remark}
Unlike the Carleman estimate in \cite[Theorem 4.1]{Gao2015Observability}, the power of $\xi$ in the last term of (\ref{carest2}) is five instead of four. This is because the explicit form of $\gamma$ on $[0,T/4]$ slightly affect the estimate of $\varphi_{t}$ in $[0,T/4]$. But this change is not a problem for proving our main controllability results.
\end{remark}
\begin{remark}\label{remark01}
From the proof of Theorem \ref{carle2}, it is not difficult to prove that Carleman inequality
(\ref{carest2}) is still valid when $r$ satisfies the boundary condition
\begin{equation}\label{103}
\left\{
		\begin{aligned}
    &r(t,0)=r(t,1)=0 &\textup{in}\ &(0,T),\\
    &r_{xx}(t,0)=r_{xx}(t,1)=0 &\textup{in}\ &(0,T).
         \end{aligned}
    \right.
\end{equation}
Indeed, in order to get the Carleman estimate with (\ref{103}), we only need to modify the estimate of the boundary terms in the proof of Theorem \ref{carle2}.
We provide a brief proof of this in Appendix A.
\end{remark}
\begin{remark}
From the proof of Theorem \ref{carle2}, it is easy  to see that we can also establish the Carleman estimate for the following stochastic system with lower-order terms:
\begin{equation}\nonumber
\dd r-r_{xxxx}\dt
= (a_2r_{xx}+a_{1}r_{x}+a_{0}r+\Phi)\dt +(b_{0}r+R)\dd B(t)\quad {in}\ Q,
\end{equation}
with suitable coefficients $a_{2},a_{1},a_{0},b_{0}$. For the conciseness of the statement, we only discuss the one without lower-order terms in this paper.	
\end{remark}

Next, the following result gives a new global Carleman estimate for (\ref{equationr}) with the drift term taking values in $H^{-2}(G)$.
\begin{theorem}\label{carle1}
There exist two positive constants $\lambda_{0}>0$ and $\mu_{0}>0$, such that the unique solution $(r,R)\in \mathcal{H}_{T}\times L^2_{\mathbb{F}}(0,T;L^2(G))$ of (\ref{equationr}) with respect to $r_{T}\in L^2_{\mathcal{F}_{T}}(\Omega;L^{2}(G))$ satisfies
    \begin{equation}\label{carest1}
    \begin{split}
	&\EE\int_{G}\lambda^3\mu^4e^{30\mu m}\theta^2(0)r^2(0)\dx
	+\EE\int_{Q}\lambda^3\mu^4\xi^3\theta^2|r_{xx}|^2\dxt
	\\
	&+\EE \int_{Q}\theta^2\lambda^{5}\mu^{6}\xi^{5}|r_{x}|^2\dxt
	+\EE\int_{Q}\lambda^7\mu^8\xi^7\theta^2r^2\dxt
	\\
	\leq
	&C\bigg(\EE \int_{Q}\theta^{2}|\phi_{0}|^2\dxt
	+\EE \int_{Q}\theta^{2}\lambda^{2}\mu^{2}\xi^{3}|\phi_{1}|^2\dxt\\
	&+\EE \int_{Q}\theta^{2}\lambda^{4}\mu^{4}\xi^{5}|\phi_{2}|^2\dxt
	+\EE \intt\int_{G_0}\theta^{2}\lambda^{7}\mu^{8}\xi^{7}|r|^2\dxt\\
	&+\EE \int_{Q}\theta^{2}\lambda^{4}\mu^{4}\xi^{5}|R|^2\dxt\bigg),
	\end{split}
	\end{equation}	
for all $\lambda\geq \lambda_{0}$ and $\mu\geq\mu_{0}$, where  $C>0$ only depends on $G$ and $G_0$.	
\end{theorem}
\begin{remark}\label{remark02}
According to Remark \ref{remark01} and the proof process of Theorem \ref{carle1}, it is not difficult to see that Carleman inequality
(\ref{carest1}) is still valid when $r$ satisfies the boundary condition (\ref{103}).
\end{remark}
\begin{remark}
In \cite{Gao2016A}, the author considered the Carleman estimate for the deterministic fourth order parabolic equation with $H^{-2}$-valued source terms.
However, we cannot simply mimic the method in \cite{Gao2016A} to obtain Theorem \ref{carle1}, because we need to establish a Carleman estimate in a stochastic case and our weight function is different from the one used in \cite{Gao2016A}.
As far as we know, the Carleman estimate (\ref{carest1})  for the stochastic fourth order parabolic equation has not been addressed in the literature.
These types of Carleman estimates for second order stochastic parabolic equations can be found in \cite{Liu2014Global} and \cite{Zhang2024New}.
\end{remark}

In the past few years, the Carleman estimates have received much attention in SPDEs.
Although there are numerous results for the global Carleman estimate for deterministic fourth order parabolic type equation (e.g., we refer to the more recent work \cite{Guerrero2019Carleman,Kassab2020Null}), people know little about the Carleman estimate for the stochastic counterpart.
As far as we know, there are few references considering Carleman estimates for stochastic fourth order parabolic type equations and the  known results are only \cite{Gao2018Global,Gao2020Null,Gao2015Observability,Lu2022Null,Zhang2024Unique}.
The references \cite{Gao2018Global,Gao2020Null,Gao2015Observability} are devoted to one-dimensional stochastic Kuramoto-Sivashinsky equations, while \cite{Lu2022Null,Zhang2024Unique} is concerned with high-dimensional stochastic fourth order parabolic equations.
We can also refer to \cite{Hernandez-Santamaria2023Global,Liao2024Stability,Liao2024Exact,Yu2022Carleman,Yuan2021Inverse,Zhang2024New,Zhang2024Unique,Zhang2024Determination}
for some Carleman estimates of SPDEs and their applications.

The remainder of the paper is organized as follows. In section 2, we first establish a new weighted identity for a stochastic fourth order parabolic operator. And then, we present the proof of Theorem \ref{carle2} and Theorem \ref{carle1}. In section 3, we present the proof of Theorem \ref{control}.

\section{The proof of Theorem \ref{carle2} and Theorem \ref{carle1}}

In this section, we shall prove the Theorem \ref{carle2} and Theorem \ref{carle1}. At first, we obtain an identity for a stochastic fourth order parabolic operator, which plays a key role in proving Theorem \ref{carle2}.
\begin{lemma}\label{identity}
Let $l\in C^{1,4}((0,T)\times \mathbb{R})$, $\mathcal{C}_{0}\in C^{0,4}((0,T)\times \mathbb{R})$, $\mathcal{C}_{1}\in C^{0,3}((0,T)\times \mathbb{R})$, $\mathcal{C}_{2}\in C^{0,2}((0,T)\times \mathbb{R})$ and also $\delta\in \mathbb{R}$. Assume that $\Upsilon$ is an $H^4(\mathbb{R})$-valued continuous semi-martingale. Set $\mathcal{L}\Upsilon=\dd \Upsilon
    +\delta \Upsilon_{xxxx}\dt$, $\vartheta=e^{l}$ and $w=\vartheta \Upsilon$. Then for a.e. $x\in \mathbb{R}$ and $\mathbb{P}$-a.s. $\omega\in\Omega$, it holds that
\begin{equation}\label{identity1}
\begin{split}
	2\mathcal{P}_2(\vartheta \mathcal{L}\Upsilon)
	=\mathcal{M}_0
	+\mathcal{M}_1
+\mathcal{A}\dt
	+\mathcal{B}
	+2\mathcal{P}_2\mathcal{P}_3\dt
	+2|\mathcal{P}_2|^2\dt  ,
\end{split}	
\end{equation}
where
\begin{equation}\nonumber
\begin{split}
&\mathcal{K}_2=6\delta l_{x}^{2},
\quad \mathcal{K}_1=12\delta (l_{x}l_{xx}),
\quad\mathcal{K}_0=\delta l_{x}^{4}-l_{t},
\\
&\mathcal{G}_3=-4\delta l_{x},\
\mathcal{G}_2=\delta(-6l_{xx}+\mathcal{C}_{2}) ,
\\
&\mathcal{G}_1=\delta(-4l_{x}^{3}+\mathcal{C}_{1}),\
\mathcal{G}_0=\delta(-6l_{x}^{2}l_{xx}+4l_{x}l_{xxx}+3l_{xx}^2+\mathcal{C}_{0}),
\\
&\mathcal{P}_2
=\delta w_{xxxx}
+\mathcal{K}_2 w_{xx}
+\mathcal{K}_0 w
+\mathcal{K}_1 w_{x},
\\
&\mathcal{P}_3
=-\delta(\mathcal{C}_{1}w_{x}
+ \mathcal{C}_{2}w_{xx}
+ \mathcal{C}_{0}w
+ l_{xxxx}w
+4 l_{xxx}w_{x}),
\\
&\mathcal{B}
=\widetilde{\mathcal{B}}
+(\mathcal{B}_{3}
+\mathcal{B}_{3*}
+\mathcal{B}_{2}
+\mathcal{B}_{2*}
+\mathcal{B}_{1}
+\mathcal{B}_{1*}
+\mathcal{B}_{0})\dt,
\\
&\widetilde{\mathcal{B}}
=[(2\delta w_{xxx}
+\mathcal{K}_{2}w_{x}
+\mathcal{K}_{2x}w
)\dd w]_{x}
-[(2\delta w_{xx}
+\mathcal{K}_{2} w)\dd w_{x}]_{x},
\\
&\mathcal{B}_{3}
=\delta (\mathcal{G}_3 |w_{xxx}|^2)_{x},
\\
&\mathcal{B}_{3*}
=2\delta( \mathcal{G}_{2}w_{xxx}w_{xx})_{x}
+2\delta( \mathcal{G}_1w_{xxx}w_{x})_{x}
+2\delta( \mathcal{G}_{0}w_{xxx}w)_{x},
\\
&\mathcal{B}_{2}
=[(\mathcal{G}_3\mathcal{K}_2
-\delta \mathcal{G}_{1}
-\delta \mathcal{G}_{2x})|w_{xx}|^2]_{x},
\\
\end{split}	
\end{equation}
\begin{equation}\nonumber
\begin{split}
&\mathcal{B}_{2*}
=2[(\mathcal{G}_3\mathcal{K}_1
 -\delta \mathcal{G}_{1x}
 -\delta \mathcal{G}_{0})w_{xx}w_{x}]_{x}
+2[(\mathcal{G}_3\mathcal{K}_0
 -\delta \mathcal{G}_{0x})w_{xx} w]_{x},
\\
&\mathcal{B}_{1}
=\{[-\mathcal{G}_3\mathcal{K}_0
-(\mathcal{G}_3\mathcal{K}_1)_{x}
+\delta \mathcal{G}_{1xx}
+\mathcal{G}_{1}\mathcal{K}_{2}
+\mathcal{G}_{2}\mathcal{K}_{1}
+2\delta \mathcal{G}_{0x}]|w_{x}|^2\}_{x},
\\
&\mathcal{B}_{1*}
=\{[-2(\mathcal{G}_3\mathcal{K}_0)_{x}
+2\mathcal{G}_{2}\mathcal{K}_{0}
+2\delta \mathcal{G}_{0xx}
+2\mathcal{G}_{0}\mathcal{K}_{2}
-\mathcal{K}_{2t}]w_{x}w\}_{x},
\\
&\mathcal{B}_{0}
=\{[(\mathcal{G}_3\mathcal{K}_0)_{xx}
+\mathcal{G}_{1}\mathcal{K}_{0}
-(\mathcal{G}_{2}\mathcal{K}_{0})_{x}
-\delta \mathcal{G}_{0xxx}
-(\mathcal{G}_{0}\mathcal{K}_{2})_{x}
+\mathcal{G}_{0}\mathcal{K}_{1}
+\frac{1}{2}\mathcal{K}_{2xt}]|w|^2\}_{x},
\\
&\mathcal{M}_0
=\delta \dd(|w_{xx}|^2)
+\dd(\mathcal{K}_{2}w_{xx}w)
-\frac{1}{2}\dd(\mathcal{K}_{2xx}w^2)
+\dd(\mathcal{K}_{0}w^2),
\\
&\mathcal{M}_{1}
=-\delta |\dd w_{xx}|^2
-\mathcal{K}_{2}\dd w_{xx}\dd w
+\frac{1}{2}\mathcal{K}_{2xx}|\dd w|^2
-\mathcal{K}_{0} |\dd w|^2,
\\
&\mathcal{A}
=\mathcal{A}_{3}
+\mathcal{A}_{2}
+\mathcal{A}_{1}
+\mathcal{A}_{0},
\\
&\mathcal{A}_3
=-\delta(2\mathcal{G}_{2}
+\mathcal{G}_{3x})|w_{xxx}|^2 ,
\\
&\mathcal{A}_2
=[
2\delta \mathcal{G}_{0}
+3\delta \mathcal{G}_{1x}
+2\mathcal{G}_{2}\mathcal{K}_{2}
-(\mathcal{G}_3\mathcal{K}_2)_{x}
-2\mathcal{G}_3\mathcal{K}_1
+\delta \mathcal{G}_{2xx}
]|w_{xx}|^2,
\\
&\mathcal{A}_1
=[-2\mathcal{G}_{0}\mathcal{K}_{2}
-(\mathcal{G}_{1}\mathcal{K}_{2})_{x}
+2\mathcal{G}_{1}\mathcal{K}_{1}
-2\mathcal{G}_{2}\mathcal{K}_{0}
+3(\mathcal{G}_3\mathcal{K}_0)_{x}
\\
&\quad\quad\ -4\delta \mathcal{G}_{0xx}
-\delta \mathcal{G}_{1xxx}
-(\mathcal{G}_{2}\mathcal{K}_{1})_{x}
+(\mathcal{G}_3\mathcal{K}_1)_{xx}
+\mathcal{K}_{2t}]|w_{x}|^2,
\\
&\mathcal{A}_0
=[2\mathcal{G}_{0}\mathcal{K}_{0}
-(\mathcal{G}_{1}\mathcal{K}_{0})_{x}
\\
&\quad\quad\ +\delta \mathcal{G}_{0xxxx}
+(\mathcal{G}_{0}\mathcal{K}_{2})_{xx}
-(\mathcal{G}_{0}\mathcal{K}_{1})_{x}
+(\mathcal{G}_{2}\mathcal{K}_{0})_{xx}
-(\mathcal{G}_3\mathcal{K}_0)_{xxx}
-\mathcal{K}_{0t}]|w|^2.
\end{split}	
\end{equation}
\end{lemma}
\begin{remark}
Since the construction of the weight function $\gamma$ makes $\theta$ not vanish at $t=0$, we need to carefully estimate all the time boundary terms in the weighted identity.
This leads us to fail in proving our Carleman estimate by using the weighted identity of the fourth order stochastic parabolic operator in \cite{Gao2018Global,Gao2015Observability}.
Therefore, we improve the previous weighted identity for stochastic fourth order parabolic operator (see e.g.,\cite{Gao2018Global,Gao2015Observability}), especially improving the time boundary terms part, so that it can be used to derive our Carleman estimates. Specifically, it can be seen that our time boundary term $\mathcal{M}_{0}$ in (\ref{identity1}) does not contain the term $w_{x}$, which is different from that in \cite{Gao2018Global,Gao2015Observability}.
\end{remark}
\begin{remark}
The weighted identity in (\ref{identity1}) is a very useful identity that can be used to prove Carleman estimates. For such kind of weighted identity and corresponding Carleman estimates for second order stochastic parabolic equations, we refer to \cite{Lu2012Carleman,Tang2009Null}.
For this identity of the fourth order stochastic parabolic equations, we refer to \cite{Gao2018Global,Gao2015Observability} in the one-dimensional case and \cite{Lu2022Null,Zhang2024Unique} in the multi-dimensional case.
\end{remark}
\begin{remark}
We note that when $\delta>0$, the identity can be used to prove Carleman estimates for forward fourth order stochastic parabolic equations. While when $\delta<0$, the identity can be used to prove Carleman estimates for backward fourth order stochastic parabolic equations.	
\end{remark}

\begin{proof}[\bf Proof of lemma \ref{identity}:]\

{\it Step 1:}
Noting $\vartheta=e^{l}$ and $w=\vartheta \Upsilon$, we have

\begin{equation}\label{2}
\begin{split}
    &\vartheta \mathcal{L}\Upsilon
    =\vartheta \mathcal{L}(\vartheta^{-1}w)\\
    &\quad\quad =\vartheta[\dd(\vartheta^{-1}w)+\delta (\vartheta^{-1}w)_{xxxx}\dt ]\\
    &\quad\quad = (\dd w -l_{t}w\dt )
    +\delta [ w_{xxxx}
    +(-4l_{x})w_{xxx}
    +6(l_{x}^{2}-l_{xx})w_{xx}\\
    &\quad \quad \quad +4(-l_{x}^{3}+3l_{x}l_{xx}-l_{xxx})w_{x}
    +(l_{x}^{4}-6l_{x}^2l_{xx}+4l_{x}l_{xxx}+3l_{xx}^2-l_{xxxx})w ]\dt.
	\\
\end{split}
\end{equation}
Set
\begin{equation}\nonumber
\begin{split}
	&\mathcal{P}_1
	=\dd w
	+[\mathcal{G}_3 w_{xxx}
	+\mathcal{G}_1 w_{x}
	+\mathcal{G}_2 w_{xx}
	+\mathcal{G}_0 w]\dt .
	\\
\end{split}
\end{equation}
Recalling the definitions of $\mathcal{P}_1$, $\mathcal{P}_2$ and $\mathcal{P}_3$, by (\ref{2}), we get
\begin{equation}\nonumber
\begin{split}
	&\vartheta \mathcal{L}\Upsilon
	=\mathcal{P}_1+\mathcal{P}_2\dt +\mathcal{P}_3 \dt ,
\end{split}
\end{equation}
which imply that
\begin{equation}\label{1}
\begin{split}
	2\mathcal{P}_2(\vartheta \mathcal{L}\Upsilon)
	=2|\mathcal{P}_2|^2\dt
	+2\mathcal{P}_1\mathcal{P}_2
	+2\mathcal{P}_2\mathcal{P}_3\dt .
\end{split}	
\end{equation}

{\it Step 2:}
Next, we will compute $2\mathcal{P}_1\mathcal{P}_2$.
We write $2\mathcal{P}_1\mathcal{P}_2$ in the form $\sum_{i=1}^{5}\sum_{j=1}^{4}2I_{ij}$, where $I_{ij}$ is the $i$-th term of $\mathcal{P}_1$ times the $j$-th term of $\mathcal{P}_2$. By It\^o's formula, $I_{11}$-$I_{14}$ can be computed as follows:
\begin{equation}\nonumber
\begin{split}
2I_{11}=\ &2(\delta w_{xxx}\dd w-\delta w_{xx}\dd w_{x})_{x}
+\delta \dd(|w_{xx}|^2)
-\delta |\dd w_{xx}|^2,
\\
2I_{12}=\ &(\mathcal{K}_{2}w_{x}\dd w
+\mathcal{K}_{2x}w\dd w
-\mathcal{K}_{2} w\dd w_{x})_{x}
+(-\mathcal{K}_{2t}ww_{x}
+\frac{1}{2}\mathcal{K}_{2xt}w^2)_{x}\dt \\
&\ +\dd(\mathcal{K}_{2}w_{xx}w)
-\frac{1}{2}\dd(\mathcal{K}_{2xx}w^2)
+\mathcal{K}_{2t}w_x^2\dt \\
&\ -\mathcal{K}_{2}\dd w_{xx}\dd w
+\frac{1}{2}\mathcal{K}_{2xx}|\dd w|^2
-2\mathcal{K}_{2x}w_{x}\dd w,
\\
2I_{13}=\ &\dd(\mathcal{K}_{0}w^2)
-\mathcal{K}_{0t}w^2\dt
-\mathcal{K}_{0} |\dd w|^2,
\\
I_{14}=\ &\mathcal{K}_{1}w_{x}\dd w.
\end{split}
\end{equation}
By direct calculation, one can calculate $I_{ij}$, $i=2,3,4,5, j=1,2,3,4$ as follows:
\begin{equation}\nonumber
\begin{split}
2I_{21}=\ &\delta (\mathcal{G}_3 |w_{xxx}|^2)_{x}\dt
-\delta \mathcal{G}_{3x}|w_{xxx}|^2\dt ,
\\
2I_{22}=\ &(\mathcal{G}_3\mathcal{K}_2 |w_{xx}|^2)_{x}\dt
-(\mathcal{G}_3\mathcal{K}_2)_{x}|w_{xx}|^2\dt ,
\\
2I_{23}=\ &2(\mathcal{G}_3\mathcal{K}_0 w_{xx} w)_{x}\dt
-2[(\mathcal{G}_3\mathcal{K}_0)_{x}w_{x}w]_{x}\dt
+[(\mathcal{G}_3\mathcal{K}_0)_{xx}|w|^2]_{x}\dt
-(\mathcal{G}_3\mathcal{K}_0 |w_{x}|^2)_{x}\dt \\
 &\ +3(\mathcal{G}_3\mathcal{K}_0)_{x}|w_{x}|^2\dt
 -(\mathcal{G}_3\mathcal{K}_0)_{xxx}|w| ^2\dt ,
 \\
 2I_{24}=\ &2(\mathcal{G}_3\mathcal{K}_1w_{xx}w_{x})_{x}\dt
 -[(\mathcal{G}_3\mathcal{K}_1)_{x}|w_{x}|^2]_{x}\dt
 -2\mathcal{G}_3\mathcal{K}_1|w_{xx}|^2\dt
 +(\mathcal{G}_3\mathcal{K}_1)_{xx}|w_{x}|^2\dt ,
\\
2I_{31}=\ &2(\delta \mathcal{G}_1w_{x}w_{xxx})_{x}\dt
-2(\delta \mathcal{G}_{1x}w_{x}w_{xx})_{x}\dt
-(\delta \mathcal{G}_{1}|w_{xx}|^2)_{x}\dt
+(\delta \mathcal{G}_{1xx}|w_{x}|^2)_{x}\dt \\
&+3\delta \mathcal{G}_{1x}|w_{xx}|^2\dt
-\delta \mathcal{G}_{1xxx}|w_{x}|^2\dt,
\\
2I_{32}=\ &(\mathcal{G}_{1}\mathcal{K}_{2}|w_{x}|^2)_{x}\dt
-(\mathcal{G}_{1}\mathcal{K}_{2})_{x}|w_{x}|^2\dt ,
\\
2I_{33}=\ &(\mathcal{G}_{1}\mathcal{K}_{0}w^2)_{x}\dt
-(\mathcal{G}_{1}\mathcal{K}_{0})_{x}w^2\dt ,
\\
I_{34}=\ &\mathcal{G}_{1}\mathcal{K}_{1}|w_{x}|^2\dt ,
\\
2I_{41}=\ &2(\delta \mathcal{G}_{2}w_{xx}w_{xxx})_{x}\dt
-(\delta \mathcal{G}_{2x}|w_{xx}|^2)_{x}\dt
-2\delta \mathcal{G}_{2}|w_{xxx}|^2\dt
+\delta \mathcal{G}_{2xx}|w_{xx}|^2\dt ,
\\
I_{42}=\ &\mathcal{G}_{2}\mathcal{K}_{2}|w_{xx}|^2\dt ,
\\
2I_{43}=\ &2(\mathcal{G}_{2}\mathcal{K}_{0}w_{x}w)_{x}\dt
-[(\mathcal{G}_{2}\mathcal{K}_{0})_{x}w^2]_{x}\dt
-2\mathcal{G}_{2}\mathcal{K}_{0}|w_{x}|^2\dt
+(\mathcal{G}_{2}\mathcal{K}_{0})_{xx}w^2\dt ,
\\
2I_{44}=\ &(\mathcal{G}_{2}\mathcal{K}_{1}|w_{x}|^2)_{x}\dt
-(\mathcal{G}_{2}\mathcal{K}_{1})_{x}|w_{x}|^2\dt ,
\\
2I_{51}=\ &2(\delta \mathcal{G}_{0}w w_{xxx})_{x}\dt
-2(\delta \mathcal{G}_{0x}w w_{xx})_{x}\dt
-2(\delta \mathcal{G}_{0}w_{x}w_{xx})_{x}\dt \\
&+2(\delta \mathcal{G}_{0xx}w w_{x})_{x}\dt
+2(\delta \mathcal{G}_{0x}|w_{x}|^2)_{x}\dt
-(\delta \mathcal{G}_{0xxx}w^2)_{x}\dt \\
&+2\delta \mathcal{G}_{0}|w_{xx}|^2\dt
-4\delta \mathcal{G}_{0xx}|w_{x}|^2\dt
+\delta \mathcal{G}_{0xxxx}w^2\dt ,
\\
2I_{52}=\ &2(\mathcal{G}_{0}\mathcal{K}_{2}ww_{x})_{x}\dt
-[(\mathcal{G}_{0}\mathcal{K}_{2})_{x}w^2]_{x}\dt
-2\mathcal{G}_{0}\mathcal{K}_{2}|w_{x}|^2\dt
+(\mathcal{G}_{0}\mathcal{K}_{2})_{xx}w^2\dt ,
\\
I_{53}=\ &\mathcal{G}_{0}\mathcal{K}_{0}w^2\dt ,
\\
2I_{54}=\ &(\mathcal{G}_{0}\mathcal{K}_{1}w^2)_{x}\dt
-(\mathcal{G}_{0}\mathcal{K}_{1})_{x}w^2\dt .
\end{split}	
\end{equation}

{\it Step 3:}
By virtue of (\ref{1}) and the fact
\begin{equation}\nonumber
2\mathcal{P}_1\mathcal{P}_2=\sum_{i=1}^{5}\sum_{j=1}^{4}2I_{ij},
\end{equation}
we obtain (\ref{identity1}).
Thus, the proof of lemma \ref{identity} is complete.
\end{proof}
Based on this pointwise weighted identity, we next give the proof for Theorem \ref{carle2}. To begin with, noting (\ref{theta}), it is easy to check that
\begin{equation}\nonumber
\ell_{x}=\lambda\mu\xi\beta_{x},\quad
\ell_{xx}=\lambda\mu^2\xi|\beta_{x}|^2
  +\lambda\mu\xi\beta_{xx},\quad
\ell_{t}=\lambda\varphi\frac{\gamma_{t}}{\gamma},\quad
\ell_{xt}=\lambda\mu\xi\beta_{x}\frac{\gamma_{t}}{\gamma},
\end{equation}
which are useful in the subsequent proof.
\begin{proof}[\bf Proof of Theorem \ref{carle2}:]
The proof will be divided into several steps.

{\it Step 1.}
We apply lemma \ref{identity} with $\delta=-1$, $\Upsilon=r$, $\vartheta=\theta$, $l=\ell$ and set
\begin{equation}\label{18}
\begin{split}
&\mathcal{C}_{2}=-2\lambda\mu^2\xi\beta_{x}^2
+6\lambda\mu\xi\beta_{xx},
\\
&\mathcal{C}_{1}=0,
\\
&\mathcal{C}_{0}=
2\lambda^3\mu^4\xi^3\beta_{x}^4
+6\lambda^3\mu^3\xi^3\beta_{x}^2\beta_{xx}
-7\lambda^2\mu^4\xi^2\beta_{x}^4
\\
&-18\lambda^2\mu^3\xi^2\beta_{x}^2\beta_{xx}
-4\lambda^2\mu^2\xi^2\beta_{x}\beta_{xxx}
-3\lambda^2\mu^2\xi^2\beta_{xx}^2,
\end{split}
\end{equation}
then by simple calculation, we have
\begin{equation}\label{4}
\begin{split}
&\mathcal{K}_2
=-6 \lambda^2\mu^2\xi^2\beta_{x}^2,
\\
&\mathcal{K}_1
=-12\lambda^2\mu^3\xi^2\beta_{x}^{3}
-12\lambda^2\mu^2\xi^2\beta_{x}\beta_{xx},
\\
&\mathcal{K}_0
=-\lambda^4\mu^4\xi^4\beta_{x}^{4}
-\ell_{t},
\\
&\mathcal{G}_3
=4 \lambda\mu\xi\beta_{x},
\\
&\mathcal{G}_2
=8\lambda\mu^2\xi\beta_{x}^{2} ,
\\
&\mathcal{G}_1
=4\lambda^3\mu^3\xi^3\beta_{x}^{3},
\\
&\mathcal{G}_0
=4\lambda^3\mu^4\xi^3\beta_{x}^4.
\end{split}	
\end{equation}
Integrating (\ref{identity1}) in $Q$, taking mathematical expectation on both sides, we conclude that
\begin{equation}\label{3}
\begin{split}
	2\EE\int_{Q}\mathcal{P}_2(\theta \mathcal{L}r)\dx
	=\ &\EE\int_{Q}\mathcal{M}_0\dx
	+\EE\int_{Q}\mathcal{M}_1\dx
	+\EE\int_{Q}\mathcal{A}\dxt
	\\
	&+\EE\int_{Q}\mathcal{B}\dx
	+2\EE\int_{Q}\mathcal{P}_2\mathcal{P}_3\dxt
	+2\EE\int_{Q}|\mathcal{P}_2|^2\dxt
	=:\sum_{i=1}^{6}\mathcal{J}_{i}.
\end{split}	
\end{equation}

{\it Step 2.}
In this step, let us estimate every term in the (\ref{3}).

{\it Estimate for $\mathcal{J}_1$.}
From the definitions of $w=\theta r$ and $\ell$, we easily see that $\lim_{t\rightarrow T-}\ell(t,\cdot)=-\infty$ and thus the term at $t=T$ vanishes. Therefore, by (\ref{4}) and  expression of $\mathcal{M}_{0}$, we have
\begin{equation}\label{5}
\begin{split}
\int_0^T\mathcal{M}_0
=\ &|w_{xx}(0)|^2
-\mathcal{K}_{2}(0)w_{xx}(0)w(0)
+[\frac{1}{2} \mathcal{K}_{2xx}(0)-\mathcal{K}_{0}(0)]|w(0)|^2,
\\
=\ &|w_{xx}(0)|^2
+[\ell_{t}(0)
+\lambda^4\mu^4\xi^4(0)\beta_{x}^{4}
+\mathcal{O}(\lambda^2)\mu^4\xi^2(0)]|w(0)|^2
\\
&+6\lambda^2\mu^2\xi^2(0)\beta_{x}^2w_{xx}(0)w(0).
\end{split}	
\end{equation}
From the explicit expression of the function $\gamma(t)$ in (\ref{gamma}), we get
\begin{equation}\nonumber
\gamma'(t)=-\frac{4\sigma}{T}\bigg(1-\frac{4t}{T}\bigg)^{\sigma-1}	\quad \forall t\in [0,T/4].
\end{equation}
Using (\ref{sigma}) and the above expression, we have
\begin{equation}\label{6}
\begin{split}
\ell_{t}(0,x)=&-\frac{4}{T}\lambda^4\mu^4e^{\mu (30m-6)}\alpha(x)\\
=&\frac{4}{T}\lambda^4\mu^4e^{\mu (30m-6)}( \mu e^{\mu(10m+10)}-e^{\mu(10m+\beta(x))})\\
\geq &c_0\lambda^4\mu^5 e^{4\mu(10m+1)}
\end{split}
\end{equation}
for all $\mu> 1$ and some constant $c_0>0$ uniform with respect to $T$.
By (\ref{5}), (\ref{6}) and Young inequality, for any $\epsilon_1>0$, we obtain that
\begin{equation}\nonumber
\begin{split}
\int_0^T\mathcal{M}_0
\geq\ &(1-\epsilon_{1})|w_{xx}(0)|^2
+[c_0\lambda^4\mu^5 e^{4\mu(10m+1)}
-C_{\epsilon_{1}}\lambda^4\mu^4e^{4\mu(10m+1)}]|w(0)|^2.
\end{split}	
\end{equation}
Set $\epsilon_{1}=\frac{1}{2}$, there exists $\mu_{1}>1$, such that for all $\mu\geq\mu_{1}>1$, we get
\begin{equation}\nonumber
\begin{split}
\int_0^T\mathcal{M}_0
\geq\ &c_{1}|w_{xx}(0)|^2
+c_1\lambda^4\mu^5 e^{4\mu(10m+1)} |w(0)|^2,
\end{split}	
\end{equation}
for some constant $c_{1}>0$ only depending on $G$ and $G'$.
Hence, we have
\begin{equation}\label{19}
\begin{split}
\mathcal{J}_{1}
=\EE\int_{Q}\mathcal{M}_{0}\dx
\geq\ &c_{1}\EE\|w_{xx}(0)\|_{L^{2}(G)}^2
+c_1\lambda^4\mu^5 e^{4\mu(10m+1)}\EE \|w(0)\|_{L^{2}(G)}^2.
\end{split}	
\end{equation}

{\it Estimate for $\mathcal{J}_2$.}
By definition of $\mathcal{M}_{1}$ and (\ref{4}), we get
\begin{equation}\nonumber
\begin{split}
&\mathcal{M}_{1}
= |\dd w_{xx}|^2
+6\lambda^2\mu^2\xi^2\beta_{x}^2\dd w_{xx}\dd w
+\big(\lambda^4\mu^4\xi^4\beta_{x}^4
+\ell_{t}
-3(\lambda^2\mu^2\xi^2\beta_{x}^2)_{xx}\big)|\dd w|^2.
\end{split}
\end{equation}
Further, using Young inequality and the fact $|\varphi_{t}|\leq \lambda^3\mu^2\xi^5$ for $(t,x)\in (0,T)\times G$, for any $\epsilon_{2}>0$, we have
\begin{equation}\nonumber
\begin{split}
&\mathcal{M}_{1}
\geq (1-\epsilon_{2})|\dd w_{xx}|^2
-C_{\epsilon_{2}}\lambda^4\mu^4\xi^5|\dd w|^2.
\end{split}
\end{equation}
Set $\epsilon_{2}=\frac{1}{2}$, it is not difficult to see that
\begin{equation}\nonumber
\mathcal{M}_{1}
\geq
-C\lambda^4\mu^4\xi^5|\dd w|^2
= -C\theta^2\lambda^4\mu^4\xi^5|\dd r|^2,
\end{equation}
which imply that
\begin{equation}\label{20}
\mathcal{J}_{2}
=\EE\int_{Q}\mathcal{M}_{1}\dx
\geq -C\EE\int_{Q}\theta^2\lambda^4\mu^4\xi^5 |R|^2\dxt ,
\end{equation}
where we have used equation (\ref{equationr}).

{\it Estimate for $\mathcal{J}_3$.}
By expression of $\mathcal{A}$ and (\ref{4}), we obtain
\begin{equation}\label{7}
\begin{split}	
\mathcal{A}
=\ &(18\lambda\mu^2\xi\beta_{x}^2
+10\lambda\mu\xi\beta_{xx})|w_{xxx}|^2
\\
&+[28\lambda^3\mu^4\xi^3\beta_{x}^4
+132\lambda^3\mu^3\xi^3\beta_{x}^2\beta_{xx}
+\mathcal{O}(\lambda^2)\mu^4\xi^2
]|w_{xx}|^2
\\
&+[28\lambda^5\mu^6\xi^5\beta_{x}^6
-36\lambda^5\mu^5\xi^5\beta_{x}^4\beta_{xx}
+\mathcal{O}(\lambda^4)\mu^6\xi^4
+A_{1} ]|w_{x}|^2
\\
&+[20\lambda^7\mu^8\xi^7\beta_{x}^8
+28\lambda^7\mu^7\xi^7\beta_{x}^6\beta_{xx}
+\mathcal{O}(\lambda^6)\mu^8\xi^6
+A_{0}
+\ell_{tt}]|w|^2,
\end{split}	
\end{equation}
where
\begin{equation}\nonumber
\begin{split}
A_{1}=\ &(4\lambda^2\mu^2\xi\beta_{x}^2
-12\lambda^2\mu\xi\beta_{xx})\varphi\frac{\gamma_{t}}{\gamma}
-24\lambda^2\mu^2\xi^2\beta_{x}^2\frac{\gamma_{t}}{\gamma},
\\
A_{0}=\ &[4\lambda^4\mu^4\xi^3\beta_{x}^4
+12\lambda^4\mu^3\xi^3\beta_{x}^2\beta_{xx}
+\mathcal{O}(\lambda^3)\mu^4\xi^2
]\varphi\frac{\gamma_{t}}{\gamma}
\\
&+[8\lambda^4\mu^4\xi^4\beta_{x}^{4}
+\mathcal{O}(\lambda^2)\mu^4\xi^2]\frac{\gamma_{t}}{\gamma}.
\end{split}
\end{equation}
Further, we estimate the terms $A_{1}$, $A_{0}$ and $\ell_{tt}$ in (\ref{7}). First, from the definition of $\gamma$, it is clear that $A_{1}$ vanishes on $(T/4,T/2)$. For $t\in(T/2,T)$, we use the fact that there exists $C>0$ such that $|\gamma_t|\leq C|\gamma|^2$. Therefore, there exists a constant $C>0$ only depending on $G, G'$ such that
\begin{equation}\label{8}
|A_1|\leq C\lambda^2\mu^3\xi^3, \quad (t,x)\in (T/2,T)\times G ,	
\end{equation}
where we have used that $|\varphi\gamma|\leq \mu\xi^2$.
For $t\in(0,T/4)$, noting the facts that $\gamma_t\leq 0$, $\varphi\leq 0$ and $\gamma\in[1,2]$, it holds that
\begin{equation}\label{9}
	A_1\geq 4\lambda^2\mu^2\xi\beta_{x}^2 |\varphi|\frac{|\gamma_t|}{\gamma}
	-C\lambda^2\mu\xi|\varphi|\frac{|\gamma_t|}{\gamma},
	\quad (t,x)\in (0,T/4)\times G.
\end{equation}
Therefore, combining (\ref{8}) and (\ref{9}), we obtain that
\begin{equation}\label{14}
\begin{split}
&\EE\int_{Q}A_{1}|w_{x}|^2\dxt
\geq
\EE\int_{0}^{T/4}\int_{G}4\lambda^2\mu^2\xi\beta_{x}^2 |\varphi|\frac{|\gamma_t|}{\gamma}|w_{x}|^2\dxt
\\
	&\ -C\EE\int_{0}^{T/4}\int_{G}\lambda^2\mu\xi|\varphi|\frac{|\gamma_t|}{\gamma}|w_{x}|^2\dxt
-C\EE\int_{0}^{T}\int_{G}\lambda^2\mu^3\xi^3|w_{x}|^2\dxt.
\end{split}	
\end{equation}

By the same analysis, we obtain that $A_{0}$ vanishes on $(T/4,T/2)$. For $t\in(T/2,T)$,  it holds that
\begin{equation}\label{10}
|A_0|\leq C\lambda^4\mu^5\xi^5, \quad (t,x)\in (T/2,T)\times G.
\end{equation}
For $t\in(0,T/4)$, noting the facts that $\mu\xi\leq|\varphi|$, $\gamma_t\leq 0$, $\varphi\leq 0$ and $\gamma\in[1,2]$, it can be readily obtained that
\begin{equation}\label{11}
A_{0}\geq 4\lambda^4\mu^4\xi^3\beta_{x}^4|\varphi|\frac{|\gamma_{t}|}{\gamma}
-C(\lambda^4\mu^3\xi^3+\lambda^3\mu^4\xi^2+\lambda^2\mu^3\xi)|\varphi|\frac{|\gamma_{t}|}{\gamma},\quad (t,x)\in (0,T/4)\times G.
\end{equation}
Therefore, combining (\ref{10}) and (\ref{11}), it holds that
\begin{equation}\label{15}
\begin{split}
\EE\int_{Q}A_{0}|w|^2\dxt
\geq \ &
\EE\int_{0}^{T/4}\int_{G}4\lambda^4\mu^4\xi^3\beta_{x}^4|\varphi|\frac{|\gamma_{t}|}{\gamma}|w|^2\dxt
-C\EE\int_{Q}\lambda^4\mu^5\xi^5|w|^2\dxt
\\
&-C\EE\int_{0}^{T/4}\int_{G}(\lambda^4\mu^3\xi^3+\lambda^3\mu^4\xi^2+\lambda^2\mu^3\xi)|\varphi|\frac{|\gamma_{t}|}{\gamma}|w|^2\dxt.
\end{split}	
\end{equation}

For the term $\ell_{tt}$, we argue as follows. First, it is clear that $\ell_{tt}$ vanishes on $(T/4,T/2)$. For $t\in (0,T/4)$, using the definition of $\gamma$, it is easy to see that
$|\gamma_{tt}|\leq C\lambda^6\mu^8e^{2\mu(30m-6)}$, thus
\begin{equation}\label{12}
\begin{split}
|\ell_{tt}|=\lambda|\gamma_{tt}||\alpha(x)|
\leq \ &C \lambda^6\mu^9e^{2\mu(30m-6)} e^{\mu(10m+10)}
\\
\leq \ &C\lambda^6\mu^9\xi^7 e^{-2\mu}
\leq C\lambda^6\mu^6\xi^7,
\end{split}
\end{equation}
where we have used that $\mu^3e^{-2\mu}<\frac{1}{2}$ for all $\mu>1$.
For $t\in (T/2,T)$, noting the facts that $|\gamma_{tt}|\leq C\gamma^3$ and $|\varphi\gamma|\leq \mu\xi^2$, we get
\begin{equation}\label{13}
|\ell_{tt}|= \lambda\frac{|\gamma_{tt}|}{\gamma}|\varphi|
\leq C\lambda\mu\xi^{3}.
\end{equation}
Hence, using (\ref{12}) and (\ref{13}), we obtain that
\begin{equation}\label{16}
\begin{split}
\EE\int_{Q}\ell_{tt}|w|^2\dxt
\geq
-C\EE\int_{Q}(\lambda^6\mu^6\xi^7
+\lambda\mu\xi^{3})|w|^2\dxt.
\end{split}	
\end{equation}

Integrating (\ref{7}) in $Q$, taking mathematical expectation on both sides, using (\ref{14}), (\ref{15}), (\ref{16}), we conclude that
\begin{equation}\label{21}
\begin{split}
\mathcal{J}_{3}
=\ &\EE\int_{Q}\mathcal{A}\dxt
\\
\geq\ &18\EE\int_{Q}\lambda\mu^2\xi\beta_{x}^2|w_{xxx}|^2\dxt
-C\EE\int_{Q}\lambda\mu\xi|w_{xxx}|^2\dxt
\\
&+28\EE\int_{Q}\lambda^3\mu^4\xi^3\beta_{x}^4|w_{xx}|^2\dxt
-C\EE\int_{Q}(\lambda^3\mu^3\xi^3
+\lambda^2\mu^4\xi^2)
|w_{xx}|^2\dxt
\\
&+28\EE\int_{Q}\lambda^5\mu^6\xi^5\beta_{x}^6|w_{x}|^2\dxt
-C\EE\int_{Q}(\lambda^5\mu^5\xi^5+\lambda^4\mu^6\xi^4)
]|w_{x}|^2\dxt
\\
&+20\EE\int_{Q}\lambda^7\mu^8\xi^7\beta_{x}^8|w|^2\dxt
-C\EE\int_{Q}(\lambda^7\mu^7\xi^7
+\lambda^6\mu^8\xi^6)
|w|^2\dxt
\\
&
+4\EE\int_{0}^{T/4}\int_{G}\lambda^2\mu^2\xi\beta_{x}^2 |\varphi|\frac{|\gamma_t|}{\gamma}|w_{x}|^2\dxt
-C\EE\int_{0}^{T/4}\int_{G}\lambda^2\mu\xi|\varphi|\frac{|\gamma_t|}{\gamma}|w_{x}|^2\dxt
\\
&
+4\EE\int_{0}^{T/4}\int_{G}\lambda^4\mu^4\xi^3\beta_{x}^4|\varphi|\frac{|\gamma_{t}|}{\gamma}|w|^2\dxt
\\
&-C\EE\int_{0}^{T/4}\int_{G}(\lambda^4\mu^3\xi^3+\lambda^3\mu^4\xi^2)|\varphi|\frac{|\gamma_{t}|}{\gamma}|w|^2\dxt.
\end{split}	
\end{equation}

{\it Estimate for $\mathcal{J}_4$.}
Noting that boundary conditions of $r$ in (\ref{equationr}) and $w=\theta r$, we get
\begin{equation}\label{17}
\begin{split}
&w(0,t)=w(1,t)
=w_{x}(0,t)=w_{x}(1,t)=0,\quad
t\in (0,T).	
\end{split}
\end{equation}
By the definition of ${\mathcal{B}}$, using (\ref{17}), (\ref{4}) and Cauchy inequality, we obtain that
\begin{equation}\nonumber
\begin{split}
\mathcal{J}_{4}
=\ &\EE\int_{Q}\mathcal{B}\dx
\\
\geq\ &-4\EE\intt
 \lambda\mu\xi(1)\beta_{x}(1) |w_{xxx}(1)|^2\dt
 +4\EE\intt
 \lambda\mu\xi(0)\beta_{x}(0) |w_{xxx}(0)|^2\dt
 \\
 &-20\EE\intt \lambda^3\mu^3\xi^3(1)\beta_{x}^3(1)|w_{xx}(1)|^2 \dt
+20\EE\intt \lambda^3\mu^3\xi^3(0)\beta_{x}^3(0)|w_{xx}(0)|^2 \dt
\\
&-C\EE\intt
 \mu |w_{xxx}(1)|^2\dt
-C\EE\intt \lambda^2\mu^3\xi^2(1)|w_{xx}(1)|^2 \dt
\\
&-C\EE\intt
 \mu |w_{xxx}(0)|^2\dt
-C\EE\intt \lambda^2\mu^3\xi^2(0)|w_{xx}(0)|^2 \dt.
\end{split}	
\end{equation}
Further, noting that $\beta_{x}(1)<0$ and $\beta_{x}(0)>0$, using (\ref{beta}), we get
\begin{equation}\label{22}
\begin{split}
\mathcal{J}_{4}
=\ &\EE\int_{Q}\mathcal{B}\dx
\\
\geq\ &4\alpha_{0} \EE\intt
 \lambda\mu\xi(1) |w_{xxx}(1)|^2\dt
 +4\alpha_{0}\EE\intt
 \lambda\mu\xi(0) |w_{xxx}(0)|^2\dt
 \\
 &+20\alpha_{0}^3\EE\intt \lambda^3\mu^3\xi^3(1)|w_{xx}(1)|^2 \dt
+20\alpha_{0}^3\EE\intt \lambda^3\mu^3\xi^3(0)|w_{xx}(0)|^2 \dt
\\
&-C\EE\intt
 \mu |w_{xxx}(1)|^2\dt
-C\EE\intt \lambda^2\mu^3\xi^2(1)|w_{xx}(1)|^2 \dt
\\
&-C\EE\intt
 \mu |w_{xxx}(0)|^2\dt
-C\EE\intt \lambda^2\mu^3\xi^2(0)|w_{xx}(0)|^2 \dt.
\end{split}	
\end{equation}

{\it Estimate for $\mathcal{J}_5$.}
By the definition of $\mathcal{P}_{3}$ and (\ref{18}), we have
\begin{equation}\nonumber
\begin{split}
\EE\int_{Q}|\mathcal{P}_3|^2\dxt
\leq \ &C\EE\int_{Q}\lambda^2\mu^4\xi^2 |w_{xx}|^2\dxt
\\
&+C\EE\int_{Q}\lambda^2\mu^6\xi^2 |w_{x}|^2\dxt
+ C\EE\int_{Q}\lambda^6\mu^8\xi^6|w|^2\dxt.
\\	
\end{split}
\end{equation}
Hence, it is not difficult to see that
\begin{equation}\label{23}
\begin{split}
\mathcal{J}_{5}
=\ &2\EE\int_{Q}\mathcal{P}_2\mathcal{P}_3\dxt
\\	
\geq\ &
 -\EE\int_{Q}|\mathcal{P}_2|^2\dxt
-\EE\int_{Q}|\mathcal{P}_3|^2\dxt
\\
\geq\ & -\EE\int_{Q}|\mathcal{P}_2|^2\dxt
-C\EE\int_{Q}\lambda^2\mu^4\xi^2 |w_{xx}|^2\dxt
\\
&-C\EE\int_{Q}\lambda^2\mu^6\xi^2 |w_{x}|^2\dxt
- C\EE\int_{Q}\lambda^6\mu^8\xi^6|w|^2\dxt.
\end{split}
\end{equation}

{\it Step 3.}
Combining (\ref{19}), (\ref{20}), (\ref{21}), (\ref{22}), (\ref{23}) with  (\ref{3}), we conclude that
\begin{equation}\label{24}
\begin{split}
&\lambda^4\mu^5 e^{4\mu(10m+1)}\EE \|w(0,\cdot)\|_{L^{2}(G)}^2
+\EE\int_{Q}\lambda\mu^2\xi\beta_{x}^2|w_{xxx}|^2\dxt
\\
&+\EE\int_{Q}\lambda^3\mu^4\xi^3\beta_{x}^4|w_{xx}|^2\dxt
+\EE\int_{Q}\lambda^5\mu^6\xi^5\beta_{x}^6|w_{x}|^2\dxt
\\
&+\EE\int_{Q}\lambda^7\mu^8\xi^7\beta_{x}^8|w|^2\dxt
+\EE\int_{0}^{T/4}\int_{G}\lambda^2\mu^2\xi\beta_{x}^2 |\varphi||\gamma_t||w_{x}|^2\dxt
\\
&
+\EE\int_{0}^{T/4}\int_{G}\lambda^4\mu^4\xi^3\beta_{x}^4|\varphi||\gamma_t||w|^2\dxt
+ \EE\intt
 \lambda\mu\xi(t,1) |w_{xxx}(t,1)|^2\dt
\\
&+\EE\intt
 \lambda\mu\xi(t,0) |w_{xxx}(t,0)|^2\dt
+\EE\intt \lambda^3\mu^3\xi^3(t,1)|w_{xx}(t,1)|^2 \dt
\\
&+\EE\intt \lambda^3\mu^3\xi^3(t,0)|w_{xx}(t,0)|^2 \dt
+\EE\int_{Q}|\mathcal{P}_2|^2\dxt
	\\
\leq\
&C\EE\int_{Q}\mathcal{P}_2(\theta \mathcal{L}r)\dx
+C\EE\int_{Q}\theta^2\lambda^4\mu^4\xi^5 |R|^2\dxt
+C\mathcal{X}_{1},
\end{split}	
\end{equation}
where
\begin{equation}\nonumber
\begin{split}
\mathcal{X}_{1}
=\ &\EE\int_{Q}\lambda\mu\xi|w_{xxx}|^2\dxt
+\EE\int_{Q}(\lambda^3\mu^3\xi^3
+\lambda^2\mu^4\xi^2)
|w_{xx}|^2\dxt
\\
&+\EE\int_{Q}(\lambda^5\mu^5\xi^5+\lambda^4\mu^6\xi^4)
]|w_{x}|^2\dxt
+\EE\int_{Q}(\lambda^7\mu^7\xi^7
+\lambda^6\mu^8\xi^6)
|w|^2\dxt
\\
&+\EE\int_{0}^{T/4}\int_{G}\lambda^2\mu\xi|\varphi||\gamma_t||w_{x}|^2\dxt
\\
&+\EE\int_{0}^{T/4}\int_{G}(\lambda^4\mu^3\xi^3+\lambda^3\mu^4\xi^2)|\varphi||\gamma_t||w|^2\dxt
\\
&+\EE\intt
 \mu |w_{xxx}(t,1)|^2\dt
+\EE\intt \lambda^2\mu^3\xi^2(t,1)|w_{xx}(t,1)|^2 \dt
\\
&+\EE\intt
 \mu |w_{xxx}(t,0)|^2\dt
+\EE\intt \lambda^2\mu^3\xi^2(t,0)|w_{xx}(t,0)|^2 \dt
\end{split}	
\end{equation}
and for some $C > 0$ only depending on $G$, $G'$ and $\alpha_{0}$.
Using Young inequality and (\ref{equationr}), noting that $\phi_{1}\equiv 0$, $\phi_{2}\equiv 0$ in (\ref{equationr}), for any $\epsilon_{3}>0$, we have
\begin{equation}\label{25}
\begin{split}
\EE\int_{Q}\mathcal{P}_2(\theta \mathcal{L}r)\dx
\leq
\epsilon_{3}\EE\int_{Q}|\mathcal{P}_2|^2\dxt
+C_{\epsilon_{3}}\EE\int_{Q}\theta^2|\phi_{0}|^2\dxt .
\end{split}
\end{equation}
Set $\epsilon_{3}=\frac{1}{2}$, noting that $\inf_{\overline{G\backslash G' }}\{|\beta_{x}(x)|\}\geq \alpha_{0}$, we can combine estimate (\ref{24}) with  (\ref{25}) to deduce
\begin{equation}\nonumber
\begin{split}
&\lambda^4\mu^5 e^{4\mu(10m+1)}\EE \|w(0,\cdot)\|_{L^{2}(G)}^2
+\EE\int_{Q}\lambda\mu^2\xi|w_{xxx}|^2\dxt
\\
&+\EE\int_{Q}\lambda^3\mu^4\xi^3|w_{xx}|^2\dxt
+\EE\int_{Q}\lambda^5\mu^6\xi^5|w_{x}|^2\dxt
\\
&+\EE\int_{Q}\lambda^7\mu^8\xi^7|w|^2\dxt
+\EE\int_{0}^{T/4}\int_{G}\lambda^2\mu^2\xi |\varphi||\gamma_t||w_{x}|^2\dxt
\\
&+\EE\int_{0}^{T/4}\int_{G}\lambda^4\mu^4\xi^3|\varphi||\gamma_t||w|^2\dxt
+\EE\intt
 \lambda\mu\xi(t,1) |w_{xxx}(t,1)|^2\dt
\\
&+\EE\intt
 \lambda\mu\xi(t,0) |w_{xxx}(t,0)|^2\dt
+\EE\intt \lambda^3\mu^3\xi^3(t,1)|w_{xx}(t,1)|^2 \dt
\\
&+\EE\intt \lambda^3\mu^3\xi^3(t,0)|w_{xx}(t,0)|^2 \dt
+\EE\int_{Q}|\mathcal{P}_2|^2\dxt
\\
\leq\
&C\EE\int_{Q}\theta^2|\phi_{0}|^2\dxt
+C\EE\int_{Q}\theta^2\lambda^4\mu^4\xi^5 |R|^2\dxt
+C\mathcal{X}_{1}
+C\mathcal{X}_{2},
\end{split}
\end{equation}
where
\begin{equation}\nonumber
\begin{split}
\mathcal{X}_{2}
=\ &\EE\intt\int_{G'}\lambda\mu^2\xi|w_{xxx}|^2\dxt
+\EE\intt\int_{G'}\lambda^3\mu^4\xi^3|w_{xx}|^2\dxt
\\
&+\EE\intt\int_{G'}\lambda^5\mu^6\xi^5|w_{x}|^2\dxt
+\EE\intt\int_{G'}\lambda^7\mu^8\xi^7|w|^2\dxt
\\
&+\EE\int_{0}^{T/4}\int_{G'}\lambda^2\mu^2\xi |\varphi||\gamma_t||w_{x}|^2\dxt
+\EE\int_{0}^{T/4}\int_{G'}\lambda^4\mu^4\xi^3|\varphi||\gamma_t||w|^2\dxt
\end{split}	
\end{equation}
and $C > 0$ only depending on $G$, $G'$ and $\alpha_{0}$. We observe that, unlike the traditional Carleman estimate with weight vanishing at $t=0$ and $t=T$, we have six local integrals, two of those being only for $t\in(0,T/4)$. We will handle this in the next step.

Noting that all of the terms in $\mathcal{X}_{1}$ have lower powers of $\lambda$ and $\mu$, thus we obtain that there exists a constant $\mu_{2}>\mu_{1}$ such that for all $\mu\geq \mu_{2}$, there exists a constant $\lambda_{1}=\lambda_{1}(\mu)$, such that for any $\lambda>\lambda_{1}$, it holds that
\begin{equation}\label{26}
\begin{split}
&\lambda^4\mu^5 e^{4\mu(10m+1)}\EE \| w(0,\cdot)\|_{L^{2}(G)}^2
+\EE\int_{Q}\lambda\mu^2\xi|w_{xxx}|^2\dxt
\\
&+\EE\int_{Q}\lambda^3\mu^4\xi^3|w_{xx}|^2\dxt
+\EE\int_{Q}\lambda^5\mu^6\xi^5|w_{x}|^2\dxt
\\
&+\EE\int_{Q}\lambda^7\mu^8\xi^7|w|^2\dxt
+\EE\int_{0}^{T/4}\int_{G}\lambda^2\mu^2\xi |\varphi||\gamma_t||w_{x}|^2\dxt
\\
&+\EE\int_{0}^{T/4}\int_{G}\lambda^4\mu^4\xi^3|\varphi||\gamma_t||w|^2\dxt
+\EE\intt
 \lambda\mu\xi(t,1) |w_{xxx}(t,1)|^2\dt
\\
&+\EE\intt
 \lambda\mu\xi(t,0) |w_{xxx}(t,0)|^2\dt
+\EE\intt \lambda^3\mu^3\xi^3(t,1)|w_{xx}(t,1)|^2 \dt
\\
&+\EE\intt \lambda^3\mu^3\xi^3(t,0)|w_{xx}(t,0)|^2 \dt
+\EE\int_{Q}|\mathcal{P}_2|^2\dxt
\\
\leq\
&C\EE\int_{Q}\theta^2|\phi_{0}|^2\dxt
+C\EE\int_{Q}\theta^2\lambda^4\mu^4\xi^5 |R|^2\dxt
+C\mathcal{X}_{2}.
\end{split}	
\end{equation}

{\it Step 4.}
The last step is to remove the local term containing the gradient of the solution and the local term in $(0,T/4)$ and return to the original variable.

First, using $r=\theta^{-1}w$,  it is not difficult to see that
\begin{equation}\nonumber
\begin{split}
&\theta^2|r_{xxx}|^2
\leq C(|w_{xxx}|^2
+\lambda^2\mu^2\xi^2 |w_{xx}|^2
+\lambda^4\mu^4\xi^4 |w_{x}|^2
+\lambda^6\mu^6\xi^6 |w|^2),
\\
&\theta^2|r_{xx}|^2
\leq C(|w_{xx}|^2
+\lambda^2\mu^2\xi^2 |w_{x}|^2
+\lambda^4\mu^4\xi^4 |w|^2),
\\
&\theta^2|r_{x}|^2
\leq C(|w_{x}|^2
+\lambda^2\mu^2\xi^2 |w|^2)
\end{split}	
\end{equation}
for some $C>0$ only depending on $G$ and $G'$, hence from (\ref{26}), after removing some unnecessary terms, we obtain that
\begin{equation}\label{27}
\begin{split}
&\lambda^4\mu^5 e^{4\mu(10m+1)}\EE \|\theta(0,\cdot) r(0,\cdot)\|_{L^{2}(G)}^2
+\EE\int_{Q}\theta^2\lambda\mu^2\xi|r_{xxx}|^2\dxt
\\
&+\EE\int_{Q}\theta^2\lambda^3\mu^4\xi^3|r_{xx}|^2\dxt
+\EE\int_{Q}\theta^2\lambda^5\mu^6\xi^5|r_{x}|^2\dxt
\\
&+\EE\int_{Q}\theta^2\lambda^7\mu^8\xi^7|r|^2\dxt
+\EE\int_{0}^{T/4}\int_{G}\theta^2\lambda^2\mu^2\xi |\varphi||\gamma_t||r_{x}|^2\dxt
\\
&+\EE\int_{0}^{T/4}\int_{G}\theta^2\lambda^4\mu^4\xi^3|\varphi||\gamma_t||r|^2\dxt
\\
\leq\
&C\EE\int_{Q}\theta^2|\phi_{0}|^2\dxt
+C\EE\int_{Q}\theta^2\lambda^4\mu^4\xi^5 |R|^2\dxt
+C\mathcal{X}_{3},
\end{split}	
\end{equation}
where
\begin{equation}\nonumber
\begin{split}
\mathcal{X}_{3}
=\ &\EE\intt\int_{G'}\theta^2\lambda\mu^2\xi|r_{xxx}|^2\dxt
+\EE\intt\int_{G'}\theta^2\lambda^3\mu^4\xi^3|r_{xx}|^2\dxt
\\
&+\EE\intt\int_{G'}\theta^2\lambda^5\mu^6\xi^5|r_{x}|^2\dxt
+\EE\intt\int_{G'}\theta^2\lambda^7\mu^8\xi^7|r|^2\dxt
\\
&+\EE\int_{0}^{T/4}\int_{G'}\theta^2\lambda^2\mu^2\xi |\varphi||\gamma_t||r_{x}|^2\dxt
+\EE\int_{0}^{T/4}\int_{G'}\theta^2\lambda^4\mu^4\xi^3|\varphi||\gamma_t||r|^2\dxt.
\end{split}	
\end{equation}

Next, we choose a cut-off function $\eta\in C_{0}^{\infty}(G)$ such that
\begin{equation}\nonumber
0\leq \eta \leq 1,\quad
\eta\equiv 1\ \mbox{in}\ G',\quad
\eta\equiv 0\ \mbox{in}\ G\backslash G_{0},
\end{equation}
with the additional characteristic
$
{|\eta_{x}|}/{\eta^{\frac{1}{2}}}\in L^{\infty}(G).
$
By It\^o's formula, we calculate
\begin{equation}\nonumber
\begin{split}
\theta^2\xi\eta r_{xx}\dd r
=\ &(\theta^2\xi\eta r_{x}\dd r)_{x}
-(\theta^2\xi\eta)_{x}r_{x}\dd r
-\theta^2\xi\eta r_{x}\dd r_{x}
\\
=\ &(\theta^2\xi\eta r_{x}\dd r)_{x}
-(\theta^2\xi\eta)_{x}r_{x}\dd r
-\frac{1}{2}\dd(\theta^2\xi\eta r_{x}^2)
\\
&+\frac{1}{2}(\theta^2\xi)_{t}\eta r_{x}^2\dt
+\frac{1}{2}\theta^2\xi\eta|\dd r_{x}|^2
\end{split}
\end{equation}
and thus, using the equation (\ref{equationr}) and $\eta\in C^{\infty}_{0}(G)$, integrating it over $Q$ and taking the mathematical expectation, we have
\begin{equation}\label{28}
\begin{split}
& 2\EE\int_{Q} \theta^2\xi\eta|r_{xxx}|^2\dxt
+\EE\int_{Q} (\theta^2\xi\eta)_{xxxx}|r_{x}|^2\dxt
+\EE\int_{Q} (\theta^2\xi)_{t}\eta |r_{x}|^2\dxt
\\
&+\int_{G}\theta^{2}(0,x)\xi(0,x)\eta(x) |r_{x}(0,x)|^2\dx
+\EE\int_{Q} \theta^2\xi\eta |R_{x}|^2\dxt
\\
=\ & 5\EE\int_{Q} (\theta^2\xi\eta)_{xx}|r_{xx}|^2\dxt
+2\EE\int_{Q} \theta^2\xi\eta  r_{xx}\phi_{0}\dxt
+2\EE\int_{Q} (\theta^2\xi\eta)_{x}r_{x}\phi_{0}\dt  .
\end{split}
\end{equation}

We readily see that the last two terms on the left-hand side of (\ref{28}) are positive, so they can be dropped. Also, notice that using the properties of $\eta$ and simple calculation, the first and second terms give the local terms containing $r_{xxx}$ and $r_{x}$ and other terms with lower powers of $\lambda$ and $\mu$. For the third term on the left-hand side of (\ref{28}), similar to step 2 above, we can analyze it on different time intervals. In summary, multiplying  both sides of (\ref{28}) by $\lambda\mu^2$, by the properties of $\eta$ and simple calculation, we have
\begin{equation}\label{29}
\begin{split}
& \EE\intt\int_{G'} \theta^2\lambda\mu^2\xi|r_{xxx}|^2\dxt
+\EE\intt\int_{G'}\theta^2\lambda^5\mu^6\xi^5|r_{x}|^2\dxt
\\
&+\EE\int_{0}^{T/4}\int_{G'}\theta^2\lambda^2\mu^2\xi|\varphi||\gamma_{t}| |r_{x}|^2\dxt
\\
\leq
\ &C\bigg[\EE\int_{Q}\theta^2\lambda^2\mu^4\xi^2|r_{xx}|^2\dxt
+\EE\int_{Q}\theta^2\lambda^4\mu^6\xi^4|r_{x}|^2\dxt
\\
&+\EE\int_{0}^{T/4}\int_{G}\theta^2\lambda\mu^2\xi|\varphi||\gamma_{t}| |r_{x}|^2\dxt\bigg]
+C\EE\int_{Q}\theta^2\lambda^3\mu^4\xi^3\eta|r_{xx}|^2\dxt
\\
&+C\EE\int_{Q}\theta^2|\phi_{0}|^2\dxt.
\end{split}
\end{equation}

In addition, by It\^o's formula, we calculate
\begin{equation}\nonumber
\begin{split}
\theta^2\xi^3\eta r\dd r
=
\frac{1}{2}\dd(\theta^2\xi^3\eta  r^2)
-\frac{1}{2}(\theta^2\xi^3)_{t}\eta r^2\dt
-\frac{1}{2}\theta^2\xi^3\eta|\dd r|^2.
\end{split}	
\end{equation}
Also using the equation (\ref{equationr}) and $\eta\in C^{\infty}_{0}(G)$, integrating the above expression on $Q$ and taking the mathematical expectation, we obtain that
\begin{equation}\label{30}
\begin{split}
&2\EE \int_{Q}\theta^2\xi^3\eta |r_{xx}|^2\dxt
+\EE\int_{Q}(\theta^2\xi^3)_{t}\eta r^2\dxt
+\EE \int_{Q}(\theta^2\xi^3\eta)_{xxxx}|r|^2\dxt
\\
&+\EE\int_{Q}\theta^2\xi^3\eta |R|^2\dxt
+\EE\int_{G}\theta^2(0,x)\xi^3(0,x)\eta  |r(0,x)|^2\dx
\\
=\ &4\EE \int_{Q}(\theta^2\xi^3\eta)_{xx}|r_{x}|^2\dxt
 -2\EE \int_{Q}\theta^2\xi^3\eta r\phi_{0}\dxt.
\end{split}	
\end{equation}

We immediately see that the last three terms on the left-hand side of (\ref{30}) are positive, so they can be dropped.
Similar to the analysis of (\ref{28}), multiplying both sides of (\ref{30}) by $\lambda^3\mu^4$, by the properties of $\eta$ and simple estimate, we obtain that
\begin{equation}\label{31}
\begin{split}
&\EE \int_{Q}\theta^2\lambda^3\mu^4\xi^3\eta |r_{xx}|^2\dxt
+\EE \intt\int_{G'}\theta^2\lambda^3\mu^4\xi^3 |r_{xx}|^2\dxt
\\
&+\EE\int_{0}^{T/4}\int_{G'}\theta^2\lambda^4\mu^4\xi^3|\gamma_{t}||\varphi|r^2\dxt
\\
\leq \ &
C\EE\int_{Q}\theta^2\lambda^5\mu^6\xi^5\eta|r_{x}|^2\dxt
+C\EE\int_{Q}\theta^2|\phi_{0}|^2\dxt
\\
&+C\bigg[\EE\int_{Q}\theta^2\lambda^4\mu^6\xi^4|r_{x}|^2\dxt
+\EE\int_{Q}\theta^2\lambda^6\mu^8\xi^6|r|^2\dxt \bigg].
\end{split}	
\end{equation}
By integration by parts,
\begin{equation}\label{32}
\begin{split}
\EE\int_{Q}\theta^2\lambda^5\mu^6\xi^5\eta |r_{x}|^2\dxt
=
-\EE\int_{Q}\theta^2\lambda^5\mu^6\xi^5\eta r_{xx}r\dxt
+\frac{1}{2}\EE\int_{Q}(\theta^2\lambda^5\mu^6\xi^5\eta)_{xx}r^2\dxt .
\end{split}	
\end{equation}
Using Young inequality, the properties of $\eta$ and simple estimate, for any $\epsilon_{4}>0$, (\ref{32}) deduce to
\begin{equation}\label{33}
\begin{split}
&\EE\int_{Q}\theta^2\lambda^5\mu^6\xi^5\eta |r_{x}|^2\dxt
\leq \epsilon_{4}\EE\int_{Q}\theta^2\lambda^3\mu^4\xi^3\eta |r_{xx}|^2\dxt
\\
&\quad +C_{\epsilon_{4}}\EE\intt\int_{G_{0}}\theta^2\lambda^7\mu^8\xi^7 r^2\dxt
+C\EE\int_{Q}\theta^2\lambda^6\mu^8\xi^6 r^2\dxt.
\end{split}
\end{equation}
Set $\epsilon_{4}=\frac{1}{2}$, by (\ref{31}) and (\ref{33}), we obtain
\begin{equation}\label{34}
\begin{split}
	&\EE \int_{Q}\theta^2\lambda^3\mu^4\xi^3\eta |r_{xx}|^2\dxt
+\EE \intt\int_{G'}\theta^2\lambda^3\mu^4\xi^3 |r_{xx}|^2\dxt
\\
&+\EE\int_{0}^{T/4}\int_{G'}\theta^2\lambda^4\mu^4\xi^3|\gamma_{t}||\varphi|r^2\dxt
\\
\leq \ &
C\EE\intt\int_{G_{0}}\theta^2\lambda^7\mu^8\xi^7 r^2\dxt
+C\EE\int_{Q}\theta^2|\phi_{0}|^2\dxt
\\
&+C\bigg[\EE\int_{Q}\theta^2\lambda^4\mu^6\xi^4|r_{x}|^2\dxt
+\EE\int_{Q}\theta^2\lambda^6\mu^8\xi^6|r|^2\dxt \bigg].
\end{split}	
\end{equation}

Combining (\ref{34}) and (\ref{29}), we conclude that
\begin{equation}\label{35}
\begin{split}
& \EE\intt\int_{G'} \theta^2\lambda\mu^2\xi|r_{xxx}|^2\dxt
+\EE\intt\int_{G'}\theta^2\lambda^5\mu^6\xi^5|r_{x}|^2\dxt
\\
&+\EE\int_{0}^{T/4}\int_{G'}\theta^2\lambda^2\mu^2\xi|\varphi||\gamma_{t}| |r_{x}|^2\dxt
\\
&+\EE \intt\int_{G'}\theta^2\lambda^3\mu^4\xi^3 |r_{xx}|^2\dxt
+\EE\int_{0}^{T/4}\int_{G'}\theta^2\lambda^4\mu^4\xi^3|\gamma_{t}||\varphi|r^2\dxt
\\
\leq
\ &
C\EE\intt\int_{G_{0}}\theta^2\lambda^7\mu^8\xi^7 r^2\dxt
+C\EE\int_{Q}\theta^2|\phi_{0}|^2\dxt
\\
&+C\bigg[\EE\int_{Q}\theta^2\lambda^2\mu^4\xi^2|r_{xx}|^2\dxt
+\EE\int_{Q}\theta^2\lambda^4\mu^6\xi^4|r_{x}|^2\dxt
\\
&+\EE\int_{0}^{T/4}\int_{G}\theta^2\lambda\mu^2\xi|\varphi||\gamma_{t}| |r_{x}|^2\dxt
+\EE\int_{Q}\theta^2\lambda^6\mu^8\xi^6|r|^2\dxt\bigg] .
\end{split}
\end{equation}

Using (\ref{27}) and (\ref{35}), we obtain that
\begin{equation}\label{36}
\begin{split}
&\lambda^4\mu^5 e^{4\mu(10m+1)}\EE \|\theta(0,\cdot) r(0,\cdot)\|_{L^{2}(G)}^2
+\EE\int_{Q}\theta^2\lambda\mu^2\xi|r_{xxx}|^2\dxt
\\
&+\EE\int_{Q}\theta^2\lambda^3\mu^4\xi^3|r_{xx}|^2\dxt
+\EE\int_{Q}\theta^2\lambda^5\mu^6\xi^5|r_{x}|^2\dxt
\\
&+\EE\int_{Q}\theta^2\lambda^7\mu^8\xi^7|r|^2\dxt
+\EE\int_{0}^{T/4}\int_{G}\theta^2\lambda^2\mu^2\xi |\varphi||\gamma_t||r_{x}|^2\dxt
\\
&+\EE\int_{0}^{T/4}\int_{G}\theta^2\lambda^4\mu^4\xi^3|\varphi||\gamma_t||r|^2\dxt
\\
\leq\
&C\EE\int_{Q}\theta^2|\phi_{0}|^2\dxt
+C\EE\int_{Q}\theta^2\lambda^4\mu^4\xi^5 |R|^2\dxt
+C\EE\intt\int_{G_{0}}\theta^2\lambda^7\mu^8\xi^7 r^2\dxt
\\
&+C\bigg[\EE\int_{Q}\theta^2\lambda^2\mu^4\xi^2|r_{xx}|^2\dxt
+\EE\int_{Q}\theta^2\lambda^4\mu^6\xi^4|r_{x}|^2\dxt
\\
&+\EE\int_{0}^{T/4}\int_{G}\theta^2\lambda\mu^2\xi|\varphi||\gamma_{t}| |r_{x}|^2\dxt
+\EE\int_{Q}\theta^2\lambda^6\mu^8\xi^6|r|^2\dxt\bigg].
\end{split}	
\end{equation}

Finally, from (\ref{36}), we can see that there exist two constants $\mu_{0}>\mu_{2}$, $\lambda_{0}>\lambda_{1}$, such that (\ref{carest2}) holds for all $\mu\geq\mu_{0}$, $\lambda\geq\lambda_{0}$.
Thus, the proof of Theorem \ref{carle2} is complete.
\end{proof}

To prove the Carleman estimates in Theorem \ref{carle1}, let us introduce the following auxiliary control problem:
\begin{equation}\label{equationh}%4.1
	\left\{
		\begin{aligned}
	&\dd h+h_{xxxx}\dt
	=(\lambda^7\mu^8\xi^7\theta^2r+\chi_{G_0}v)\dt+(\lambda^4\mu^4\xi^5\theta^2R+V)\dd B(t) &\textup{in}\ &Q,\\
    &h(t,0)=h(t,1)=0 &\textup{in}\ &(0,T),\\
    &h_{x}(t,0)=h_{x}(t,1)=0 &\textup{in}\ &(0,T),\\
    &h(0)=h_0 &\textup{in}\ &G,
         \end{aligned}
    \right.
\end{equation}
where $h=h(t,x)$ denotes the state variable associated to the initial state $h_{0}\in L^2_{\mathcal{F}_{0}}(\Omega;L^2(G))$, the control pair $(v,V)\in L^2_{\mathbb{F}}(0,T;L^2(G_{0}))\times L^2_{\mathbb{F}}(0,T;L^2(G))$ and $(r,R)\in \mathcal{H}_{T}\times L^2_{\mathbb{F}}(0,T;L^2(G))$ is the unique solution of (\ref{equationr}) with respect to $r_{T}\in L^2_{\mathcal{F}_{T}}(\Omega;L^2(G))$.
Observe that given the aforementioned regularity on the controls and source term, one can easily show that system $(\ref{equationh})$ admits a unique solution $h\in \mathcal{H}_{T}$ (see e.g., \cite[Proposition 2.3]{Gao2015Observability}).

By Theorem \ref{carle2}, we have the following controllability result of the system (\ref{equationh}).
\begin{lemma}\label{carle3}
Let $(r,R)$ be the unique solution to the backward stochastic fourth order parabolic equation $(\ref{equationr})$, with respect to the terminal state $r_{T}\in L^2_{\mathcal{F}_{T}}(\Omega;H_0^2(G))$. Then there exists a control pair $(\hat v,\hat V)\in L^2_{\mathbb{F}}(0,T;L^2(G_{0}))\times L^2_{\mathbb{F}}(0,T;L^2(G))$, such that the corresponding solution $\hat h$ to $(\ref{equationh})$ verifies $\hat y(T)=0$ in $G$, $\mathbb{P}$-a.s. Moreover, one can find two positive constants $\lambda_{0}$, $\mu_{0}$, such that
\begin{equation}\label{carest3}
\begin{split}
&\EE\intt\int_{G_{0}}\lambda^{-7}\mu^{-8}\xi^{-7}\theta^{-2}|v|^2\dxt
	+\EE\int_{Q}\lambda^{-4}\mu^{-4}\xi^{-5}\theta^{-2}|V|^2\dxt
	\\
	 &+\EE\int_{Q}\theta^{-2}h^2\dxt
	 +\EE\int_{Q}\lambda^{-2}\mu^{-2}\xi^{-3}\theta^{-2}|h_{x}|^2\dxt
	 \\
	 &+\EE \int_{Q}\lambda^{-4}\mu^{-4}\xi^{-5}\theta^{-2}|h_{xx}|^2\dxt
	 \\
	 \leq
	 &C \bigg(
	 \EE\int_{G}\lambda^{-4}\mu^{-5}e^{-4\mu(10m+1)}\theta^{-2}(0)|h_{0}|^2\dx
	 +\EE\int_{Q} \lambda^7\mu^8\xi^7\theta^2 r^2\dxt
	 \\
	 &+\EE\int_{Q} \lambda^4\mu^4\xi^5\theta^2 R^2\dxt
	 \bigg),
\end{split}
\end{equation}
for all $\lambda\geq \lambda_{0}$ and $\mu\geq\mu_{0}$, where  $C>0$ only depends on $G$ and $G_0$.
\end{lemma}
\begin{proof}[\bf Proof of lemma \ref{carle3}:]
We borrow some ideas from \cite{Liu2014Global}. The main steps are as follows. First, we construct a family of optimal control problems for equation (\ref{equationh}). Next, we establish a uniform estimate for these optimal solutions. Finally, by taking the limit, one can obtain the estimate (\ref{carest3}) and the desired null controllability result.
We divide the proof into three parts.

{\it Step 1.}
For any $\varepsilon>0$, consider the following weight function
\begin{equation}\nonumber
\gamma_{\varepsilon}(t)=
	\left\{
		\begin{aligned}
	&1+\bigg(1-\frac{4t}{T}\bigg)^{\sigma}, &t\in &[0,T/4),\\
	&1,&t\in &[T/4,T/2+\varepsilon),\\
	&\gamma(t-\varepsilon ),& t\in &[T/2+\varepsilon,T],
	    \end{aligned}
    \right.
\end{equation}
where $\sigma$ is the same as in $(\ref{sigma})$.
Using the new weight function $\gamma_{\varepsilon}(t)$, we set
\begin{equation}\nonumber
	\varphi_{\varepsilon}(t,x):=\gamma_{\varepsilon}(t)\alpha(x),\quad
	\theta_{\varepsilon}:=e^{\lambda\varphi_{\varepsilon} }.
\end{equation}
With this notation, we introduce the functional
\begin{equation}\nonumber
\begin{split}
	J_{\varepsilon}(v,V)
	:=&\frac{1}{2}\EE \int_{Q}\theta_{\varepsilon}^{-2}|h|^2\dxt
	+\frac{1}{2}\EE \intt\int_{G_0}\theta^{-2}\lambda^{-7}\mu^{-8}\xi^{-7}|v|^2\dxt
	\\
	&+\frac{1}{2}\EE \int_{Q}\theta^{-2}\lambda^{-4}\mu^{-4}\xi^{-5}|V|^2\dxt
	+\frac{1}{2\varepsilon}\EE \int_{G}|h(T)|^2\dx
\end{split}	
\end{equation}
and consider the following optimal  control problem :
\begin{equation}\label{44}
\begin{split}
\left\{
\begin{aligned}
	&\min_{(v,V)\in \mathscr{V}}J_{\varepsilon}(v,V)\\
	&\text{subject to equation}\ (\ref{equationh}),
\end{aligned}
\right.
\end{split}
\end{equation}
where
\begin{equation}\nonumber
\begin{split}
	\mathscr{V}=\bigg\{&(v,V)\in L^2_\mathbb{F}(0,T;L^2(G_0))\times L^2_\mathbb{F}(0,T;L^2(G))\bigg|\\
	&\EE \intt\int_{G_0}\theta^{-2}\lambda^{-7}\mu^{-8}\xi^{-7}|v|^2\dxt<\infty,\
	\EE \int_{Q}\theta^{-2}\lambda^{-4}\mu^{-4}\xi^{-5}|V|^2\dxt<+\infty \bigg\}.
\end{split}	
\end{equation}
Similar to \cite{Li1995Optimal}, it is easy to check that for any $\varepsilon>0$, (\ref{44}) admits a unique optimal pair solution that we denote by $(v_{\varepsilon},V_{\varepsilon})$. Moreover, by the standard variational method (see \cite{Li1995Optimal,Lion1971Optimal}), it follows that
\begin{equation}\label{43}
	v_{\varepsilon}=\chi _{G_0}\theta^{2}\lambda^{7}\mu^{8}\xi^{7}p_{\varepsilon},\quad
	V_{\varepsilon}=\theta^{2}\lambda^{4}\mu^{4}\xi^{5}P_{\varepsilon}\quad \text{in}\ Q,\ \mathbb{P}\text{-}a.s.,
\end{equation}
where the pair $(p_{\varepsilon},P_{\varepsilon})$ verifies the backward stochastic equation
\begin{equation}\label{equationp}
	\left\{
		\begin{aligned}
	&\dd p_{\varepsilon}-p_{\varepsilon xxxx}\dt
	=\theta_{\varepsilon}^{-2}h_{\varepsilon}\dt
	+P_{\varepsilon}\dd B(t) &\textup{in}\ &Q,\\
    &p_{\varepsilon}(t,0)=p_{\varepsilon}(t,1)=0 &\textup{in}\ &(0,T),\\
    &p_{\varepsilon x}(t,0)=p_{\varepsilon x}(t,1)=0 &\textup{in}\ &(0,T),\\
    &p_{\varepsilon}(T)=-\frac{1}{\varepsilon}h_{\varepsilon}(T) &\textup{in}\ &G,
            \end{aligned}
    \right.
\end{equation}
where
$h_{\varepsilon}$ is the solution of
(\ref{equationh}) with the controls $(v,V)=(v_{\varepsilon},V_{\varepsilon})$.

{\it Step 2.}
By applying the It\^o formula to calculate $\dd(h_{\varepsilon}p_{\varepsilon})$, integrating on $Q$ and taking mathematical expectation on both sides, using (\ref{equationh}), (\ref{43}) and (\ref{equationp}), we obtain that
\begin{equation}\label{45}
\begin{split}
	&\EE\intt\int_{G_{0}}\lambda^7\mu^8\xi^7\theta^2|p_{\varepsilon}|^2\dxt
	+\EE\int_{Q}\lambda^4\mu^4\xi^5\theta^2|P_{\varepsilon}|^2\dxt
	 \\
	 &+\EE\int_{Q}\theta_{\varepsilon}^{-2}|h_{\varepsilon}|^2\dxt
	 +\frac{1}{\varepsilon}\EE\int_{G}\theta_{\varepsilon}^{-2}|h_{\varepsilon}(T)|^2\dx\\
	=&-\EE\int_{G}{h}_{0}p_{\varepsilon}(0)\dx
	-\EE\int_{Q}\lambda^7\mu^8\xi^7\theta^2{p}_{\varepsilon}r\dxt
	-\EE\int_{Q}\lambda^4\mu^4\xi^5\theta^2{P}_{\varepsilon}R\dxt.
\end{split}
\end{equation}
By Young inequality and (\ref{45}), for any $\epsilon_{5}>0$, it follows that
\begin{equation}\label{46}
\begin{split}
	&\EE\intt\int_{G_{0}}\lambda^7\mu^8\xi^7\theta^2|p_{\varepsilon}|^2\dxt
	+\EE\int_{Q}\lambda^4\mu^4\xi^5\theta^2|P_{\varepsilon}|^2\dxt
	 \\
	 &+\EE\int_{Q}\theta_{\varepsilon}^{-2}|h_{\varepsilon}|^2\dxt
	 +\frac{1}{\varepsilon}\EE\int_{G}\theta_{\varepsilon}^{-2}|h_{\varepsilon}(T)|^2\dx\\
	 \leq
	 &\epsilon_5 \bigg(
	 \EE\int_{G}\lambda^4\mu^5e^{4\mu(10m+1)}\theta^2(0,x)|p_{\varepsilon}(0,x)|^2\dx
	 +\EE\int_{Q} \lambda^7\mu^8\xi^7\theta^2 |{p}_{\varepsilon}|^2\dxt
	 \\
	 &+\EE\int_{Q} \lambda^4\mu^4\xi^5\theta^2 |{P}_{\varepsilon}|^2\dxt
	 \bigg)
	 +C_{\epsilon_5} \bigg(
	 \EE\int_{G}\lambda^{-4}\mu^{-5}e^{-4\mu(10m+1)}\theta^{-2}(0,x)|h_{0}|^2\dx
	 \\
	 &+\EE\int_{Q} \lambda^7\mu^8\xi^7\theta^2 r^2\dxt
	 +\EE\int_{Q} \lambda^4\mu^4\xi^5\theta^2 R^2\dxt
	 \bigg).
\end{split}
\end{equation}

Applying the Carleman estimate (\ref{carest2}) in Theorem \ref{carle2} to (\ref{equationp}), we get
\begin{equation}\label{47}	
    \begin{split}
	&\EE \int_{G}\lambda^4\mu^5e^{4\mu(10m+1)}\theta^2(0,x)|p_{\varepsilon}(0,x)|^2\dx
	+\EE \int_{Q}\lambda^7\mu^8\xi^7\theta^2|p_{\varepsilon}|^2\dxt \\
	\leq &
	C\bigg(\EE \int_{0}^{T}\int_{G_0} \lambda^7\mu^8\xi^7\theta^2|p_{\varepsilon}|^2\dxt
	+\EE \int_{Q}\theta^2|\theta_{\varepsilon}^{-2}h_{\varepsilon}|^2\dxt
	+\EE \int_{Q}\lambda^4\mu^4\xi^5\theta^2|P_{\varepsilon}|^2\dxt
	\bigg)\\
	\leq &
	C\bigg(\EE \int_{0}^{T}\int_{G_0} \lambda^7\mu^8\xi^7\theta^2|p_{\varepsilon}|^2\dxt
	+\EE \int_{Q}\theta_{\varepsilon}^{-2}|h_{\varepsilon}|^2\dxt
	+\EE \int_{Q}\lambda^4\mu^4\xi^5\theta^2|P_{\varepsilon}|^2\dxt
	\bigg),
	\end{split}
	\end{equation}
where we have used $\theta^2\theta_{\varepsilon}^{-2}<1$.

Using (\ref{47}) and choosing $\epsilon_{5}>0$ small enough in (\ref{46}), we obtain that
\begin{equation}\label{49}
\begin{split}
&\EE\intt\int_{G_{0}}\lambda^7\mu^8\xi^7\theta^2|p_{\varepsilon}|^2\dxt
	+\EE\int_{Q}\lambda^4\mu^4\xi^5\theta^2|P_{\varepsilon}|^2\dxt
	 \\
	 &+\EE\int_{Q}\theta_{\varepsilon}^{-2}|h_{\varepsilon}|^2\dxt
	 +\frac{1}{\varepsilon}\EE\int_{G}\theta_{\varepsilon}^{-2}|h_{\varepsilon}(T)|^2\dx\\
	 \leq
	 &C \bigg(
	 \EE\int_{G}\lambda^{-4}\mu^{-5}e^{-4\mu(10m+1)}\theta^{-2}(0,x)|h_{0}|^2\dx
	 +\EE\int_{Q} \lambda^7\mu^8\xi^7\theta^2 r^2\dxt
	 \\
	 &+\EE\int_{Q} \lambda^4\mu^4\xi^5\theta^2 R^2\dxt
	 \bigg).
\end{split}
\end{equation}
Noting the representation of the control pair $(v_{\varepsilon},V_{\varepsilon})$ in (\ref{43}), we deduce from (\ref{49}) that
\begin{equation}\label{60}
\begin{split}
&\EE\intt\int_{G_{0}}\lambda^{-7}\mu^{-8}\xi^{-7}\theta^{-2}|v_{\varepsilon}|^2\dxt
	+\EE\int_{Q}\lambda^{-4}\mu^{-4}\xi^{-5}\theta^{-2}|V_{\varepsilon}|^2\dxt
	 \\
	 &+\EE\int_{Q}\theta_{\varepsilon}^{-2}|h_{\varepsilon}|^2\dxt
	 +\frac{1}{\varepsilon}\EE\int_{G}\theta_{\varepsilon}^{-2}|h_{\varepsilon}(T)|^2\dx\\
	 \leq
	 &C \bigg(
	 \EE\int_{G}\lambda^{-4}\mu^{-5}e^{-4\mu(10m+1)}\theta^{-2}(0,x)|h_{0}|^2\dx
	 +\EE\int_{Q} \lambda^7\mu^8\xi^7\theta^2 r^2\dxt
	 \\
	 &+\EE\int_{Q} \lambda^4\mu^4\xi^5\theta^2 R^2\dxt
	 \bigg).
\end{split}	
\end{equation}

Now, we are to establish an appropriate uniform bound for $h_{\varepsilon xx}$ and $h_{\varepsilon x}$, which will be achieved by a weighted energy estimate for (\ref{equationh}). In fact, by using It\^o's formula to  calculate $\dd(\lambda^{-4}\mu^{-4}\xi^{-5}\theta_{\varepsilon}^{-2}h_{\varepsilon}^2)$, integrating by parts and taking mathematical expectation, we have
\begin{equation}\label{48}
\begin{split}
&2\EE \int_{Q}\lambda^{-4}\mu^{-4}\xi^{-5}\theta_{\varepsilon}^{-2}|h_{\varepsilon xx}|^2\dxt
\\
=\
&\EE\int_{G}\lambda^{-4}\mu^{-4}\xi^{-5}(0,x)\theta_{\varepsilon}^{-2}(0,x)h_{0}^2\dx
+\EE \int_{Q}\lambda^{-4}\mu^{-4}(\xi^{-5}\theta_{\varepsilon}^{-2})_{t}h_{\varepsilon}^2\dxt
\\
&-4\EE \int_{Q}\lambda^{-4}\mu^{-4}(\xi^{-5}\theta_{\varepsilon}^{-2})_{xx}h_{\varepsilon}h_{\varepsilon xx}\dxt
+\EE \int_{Q}\lambda^{-4}\mu^{-4}(\xi^{-5}\theta_{\varepsilon}^{-2})_{xxxx}|h_{\varepsilon}|^2\dxt
\\
&+2\EE \int_{Q}\theta_{\varepsilon}^{-2}\theta^2 \lambda^3\mu^4\xi^2 h_{\varepsilon}  r\dxt
+2\EE \intt\int_{G_{0}}\lambda^{-4}\mu^{-4}\xi^{-5}\theta_{\varepsilon}^{-2}h_{\varepsilon} v_{\varepsilon}\dxt
\\
\end{split}	
\end{equation}
\begin{equation}\nonumber
\begin{split}
&+2\EE \int_{Q}\theta_{\varepsilon}^{-2}\theta^2RV_{\varepsilon}\dxt
+\EE \int_{Q}\theta_{\varepsilon}^{-2}\theta^4\lambda^4\mu^4\xi^5|R|^2\dxt
\\
&+\EE \int_{Q}\lambda^{-4}\mu^{-4}\xi^{-5}\theta_{\varepsilon}^{-2}|V_{\varepsilon}|^2\dxt
=:\sum_{i=1}^{9}\mathcal{I}_{i},
\end{split}	
\end{equation}
where we have used the fact that the weight function $\theta^{-2}_{\varepsilon}(t,\cdot)$ does not blow up as $t\rightarrow T$, and $\xi^{-5}(T,x)=0$.
Next, we estimate $\mathcal{I}_{i}$ $(i=1,2\dots,7)$, respectively.
Noting that $\theta_{\varepsilon}(0)=\theta(0)$ and  $m\geq 1$, we have
\begin{equation}\nonumber
M_{0}:=\frac{\xi^{-5}(0)\theta_{\varepsilon}^{-2}(0)}{\mu^{-1}e^{-4\mu(10m+1)}\theta^{-2}(0)}
\leq
\frac{\mu}{e^{6\mu}}
\leq 1	
\end{equation}
for a.e. $x\in G$ and $\mu\geq 1$.
Further,
\begin{equation}\nonumber
\begin{split}
\mathcal{I}_{1}
=\ &\EE\int_{G}M_{0}\lambda^{-4}\mu^{-5}e^{-4\mu(10m+1)}\theta^{-2}(0,x)|h_{0}|^2\dx
\\
\leq\ &
\EE\int_{G}\lambda^{-4}\mu^{-5}e^{-4\mu(10m+1)}\theta^{-2}(0,x)|h_{0}|^2\dx.
\end{split}
\end{equation}
Noting that
\begin{equation}\nonumber
(\xi^{-5}\theta_{\varepsilon}^{-2})_{t}
=-(2\lambda\varphi+5)\xi^{-5}\theta^{-2}\frac{\gamma_{t}}{\gamma}
\leq -C\lambda \xi^{-5}|\varphi||\gamma_{t}|\theta^{-2}
\leq 0
\end{equation}
on $(t,x)\in [0,T/4]\times G$ and
\begin{equation}\nonumber
\begin{split}
|(\xi^{-5}\theta_{\varepsilon}^{-2})_{t}|
\leq  C(\lambda+1)\xi^{-5}\theta_{\varepsilon}^{-2}|\varphi\gamma|
\leq C(\lambda\mu+\mu)\xi^{-3}\theta_{\varepsilon}^{-2}
\leq C\lambda\mu\xi^{-3}\theta_{\varepsilon}^{-2}
\end{split}
\end{equation}
on $(t,x)\in [T/4,T]\times G$, we have
\begin{equation}\nonumber
\begin{split}
\mathcal{I}_{2}
\leq\ &
\EE \int_{0}^{T/4}\int_{G}\lambda^{-4}\mu^{-4}(\xi^{-5}\theta_{\varepsilon}^{-2})_{t}h_{\varepsilon}^2\dxt
+\EE \int_{T/4}^{T}\int_{G}\lambda^{-4}\mu^{-4}|(\xi^{-5}\theta_{\varepsilon}^{-2})_{t}|h_{\varepsilon}^2\dxt
\\
\leq\ &\EE \int_{T/4}^{T}\int_{G}\lambda^{-3}\mu^{-3}\xi^{-3}\theta_{\varepsilon}^{-2}h_{\varepsilon}^2\dxt
\\
\leq\ &\EE \int_{Q}\theta_{\varepsilon}^{-2}h_{\varepsilon}^2\dxt.
\end{split}	
\end{equation}
Using $\theta\theta_{\varepsilon}^{-1}<1$ and Young inequality, by direct computation, we obtain that
\begin{equation}\nonumber
\begin{split}
&|\mathcal{I}_{3}|
\leq
\epsilon_{6} \EE \int_{Q}\lambda^{-4}\mu^{-4}\xi^{-5}\theta_{\varepsilon}^{-2}|h_{\varepsilon xx}|^2\dxt
+C_{\epsilon_{6}}\EE \int_{Q}\theta_{\varepsilon}^{-2}|h_{\varepsilon}|^2\dxt,
\\
&|\mathcal{I}_{4}|
\leq
C\EE \int_{Q}\theta_{\varepsilon}^{-2}|h_{\varepsilon}|^2\dxt,
\\
&|\mathcal{I}_{5}|
\leq
C\EE \int_{Q}\theta^2\lambda^{6}\mu^{8}\xi^{4}r^2\dxt
+C\EE \int_{Q}\theta_{\varepsilon}^{-2}|h_{\varepsilon}|^2\dxt,
\\
&|\mathcal{I}_{6}|
\leq
C\EE \intt\int_{G_{0}}\theta_{\varepsilon}^{-2}\lambda^{-7}\mu^{-8}\xi^{-7}| v_{\varepsilon}|^2\dxt
+C\EE \int_{Q}\theta_{\varepsilon}^{-2}|h_{\varepsilon}|^2\dxt,
\\
&|\mathcal{I}_{7}|
\leq
C \EE \int_{Q}\lambda^{4}\mu^{4}\xi^{5}\theta^2 |R|^2\dxt
+C\EE \int_{Q}\lambda^{-4}\mu^{-4}\xi^{-5}\theta^{-2}|V_{\varepsilon}|^2\dxt.
\end{split}
\end{equation}
Set $\epsilon_{6}=1$, combining the above inequalities with (\ref{48}), we arrive at the following estimate
\begin{equation}\label{61}
\begin{split}
&\EE \int_{Q}\lambda^{-4}\mu^{-4}\xi^{-5}\theta_{\varepsilon}^{-2}|h_{\varepsilon xx}|^2\dxt
\\
\leq\
&C\EE\int_{G}\lambda^{-4}\mu^{-5}e^{-4\mu(10m+1)}\theta^{-2}(0,x)|h_{0}|^2\dx
+C\EE \int_{Q}\theta_{\varepsilon}^{-2}h_{\varepsilon}^2\dxt
\\
&+C\EE \int_{Q}\theta^2\lambda^{6}\mu^{8}\xi^{4}r^2\dxt
+C\EE \intt\int_{G_{0}}\theta_{\varepsilon}^{-2}\lambda^{-7}\mu^{-8}\xi^{-7}| v_{\varepsilon}|^2\dxt
\\
&+C\EE \int_{Q}\theta^2\lambda^4\mu^4\xi^5|R|^2\dxt
+C\EE \int_{Q}\lambda^{-4}\mu^{-4}\xi^{-5}\theta^{-2}|V_{\varepsilon}|^2\dxt .
\end{split}	
\end{equation}
In addition, by  integration by parts and boundary conditions in (\ref{equationh}), we have
\begin{equation}\nonumber
\begin{split}
\EE\int_{Q}\lambda^{-2}\mu^{-2}\xi^{-3}\theta_{\varepsilon}^{-2}|h_{\varepsilon x}|^2\dxt
=\ &
-\EE\int_{Q}\lambda^{-2}\mu^{-2}\xi^{-3}\theta_{\varepsilon}^{-2}h_{\varepsilon xx}h_{\varepsilon}\dxt
\\
&+\frac{1}{2}\EE\int_{Q}\lambda^{-2}\mu^{-2}(\xi^{-3}\theta_{\varepsilon}^{-2})_{xx}|h_{\varepsilon}|^2\dxt .
\end{split}	
\end{equation}
Further, by Cauchy  inequality, we obtain that
\begin{equation}\label{62}
\begin{split}
\EE\int_{Q}\lambda^{-2}\mu^{-2}\xi^{-3}\theta_{\varepsilon}^{-2}|h_{\varepsilon x}|^2\dxt
\leq\ &
C \EE \int_{Q}\lambda^{-4}\mu^{-4}\xi^{-5}\theta_{\varepsilon}^{-2}|h_{\varepsilon xx}|^2\dxt
\\
&+C\EE \int_{Q}\theta_{\varepsilon}^{-2}|h_{\varepsilon}|^2\dxt.
\end{split}	
\end{equation}
Combining the (\ref{61}), (\ref{62}) with (\ref{60}), we conclude that
\begin{equation}\label{63}
\begin{split}
&\EE\intt\int_{G_{0}}\lambda^{-7}\mu^{-8}\xi^{-7}\theta^{-2}|v_{\varepsilon}|^2\dxt
	+\EE\int_{Q}\lambda^{-4}\mu^{-4}\xi^{-5}\theta^{-2}|V_{\varepsilon}|^2\dxt
	\\
	 &+\EE\int_{Q}\theta_{\varepsilon}^{-2}|h_{\varepsilon}|^2\dxt
	 +\EE\int_{Q}\lambda^{-2}\mu^{-2}\xi^{-3}\theta_{\varepsilon}^{-2}|h_{\varepsilon x}|^2\dxt
	 \\
	 &+\EE \int_{Q}\lambda^{-4}\mu^{-4}\xi^{-5}\theta_{\varepsilon}^{-2}|h_{\varepsilon xx}|^2\dxt
	 +\frac{1}{\varepsilon}\EE\int_{G}\theta_{\varepsilon}^{-2}|h_{\varepsilon}(T)|^2\dx\\
	 \leq
	 &C \bigg(
	 \EE\int_{G}\lambda^{-4}\mu^{-5}e^{-4\mu(10m+1)}\theta^{-2}(0,x)|h_{0}|^2\dx
	 +\EE\int_{Q} \lambda^7\mu^8\xi^7\theta^2 r^2\dxt
	 \\
	 &+\EE\int_{Q} \lambda^4\mu^4\xi^5\theta^2 R^2\dxt
	 \bigg).
\end{split}	
\end{equation}
{\it Step 3.}
Since the right-hand side of (\ref{63}) is uniform with respect to $\varepsilon$, we readily obtain that there exists $(\hat h,\hat v, \hat V)$ such that
\begin{equation}\label{64}
\begin{aligned}
	&h_{\varepsilon}\rightarrow  \hat h, &\text{weakly in}\ &L^2(\Omega\times (0,T);H_0^2(G)),
	\\
	&v_{\varepsilon}\rightarrow  \hat v,&\text{weakly in}\ &L^2(\Omega\times (0,T);L^2(G_{0})),
	\\
	&V_{\varepsilon}\rightarrow  \hat V,&\text{weakly in}\ &L^2(\Omega\times (0,T);L^2(G)).
\end{aligned}
\end{equation}
We claim that $\hat h$ is the solution to (\ref{equationh}) associated with $(\hat v, \hat V)$. In fact, suppose that $\tilde h$ is the unique solution to (\ref{equationh}) with controls $(\hat v, \hat V)$.
For any $\Lambda_{1}\in L^2_{\mathbb{F}}(0,T;H^{-2}(G))$, we consider $(\psi,\Psi)$  the unique solution to the backward equation
\begin{equation}\nonumber
\left\{
		\begin{aligned}
	&\dd \psi-\psi_{xxxx}\dt=\Lambda_{1}\dt+\Psi\dd B(t) &\textup{in}\ &Q,\\
    &\psi(t,0)=\psi(t,1)=0 &\textup{in}\ &(0,T),\\
    &\psi_{x}(t,0)=\psi_{x}(t,1)=0 &\textup{in}\ &(0,T),\\
    &\psi(T)=0 &\textup{in}\ &G.
         \end{aligned}
    \right.	
\end{equation}
Then, by applying It\^o's formula to calculate $\dd(h_{\varepsilon}\psi-\tilde h\psi)$, we obtain that
\begin{equation}\label{65}
\begin{split}
0=\EE\int_{Q}(h_{\varepsilon}-\tilde h)\Lambda_{1}\dxt
+\EE\intt\int_{G_{0}}(v_{\varepsilon}-\hat v)\psi \dxt
+\EE\int_{Q}(V_{\varepsilon}-\hat V)\Psi \dxt .
\end{split}	
\end{equation}
By taking the limit as $\varepsilon\rightarrow 0$ in (\ref{65}) and using (\ref{64}), we have
\begin{equation}\nonumber
\EE\int_{Q}(\hat h-\tilde h)\Lambda_{1}\dxt =0.
\end{equation}
Since $\Lambda_{1}\in L^2_{\mathbb{F}}(0,T;H^{-2}(G))$ is arbitrary, we deduce that $\hat h=\tilde h$ in $Q$, $\mathbb{P}$-a.s.

To conclude, we notice from (\ref{63}) that $\hat h(T)=0$ in $G$, $\mathbb{P}$-a.s. Moreover, from the weak convergence (\ref{64}), the uniform estimate (\ref{63}) and Fatou's lemma, we deduce (\ref{carest3}). Thus, the proof of lemma \ref{carle3} is complete.
\end{proof}

Now, by applying the lemma \ref{carle3}, we can give the proof of Theorem \ref{carle1}.
\begin{proof}[\bf Proof of Theorem \ref{carle1}:]
Let $(r,R)$ be the solution of $(\ref{equationr})$ and $\hat h$ be the solution of $(\ref{equationh})$ associated the control pair $(\hat{v}, \hat{V})$ obtained in lemma \ref{carle3}. Set the initial state $h_{0}\equiv0$ in $(\ref{equationh})$. Using It\^o formula to calculate  $\dd(\hat{h}r)$, noting the fact $\hat {h}(T,x)=0$ in $G$, $\mathbb{P}$-a.s., integrating on $Q$ and taking mathematical expectation on both sides, we have
\begin{equation}\nonumber
\begin{split}
	&\EE\int_{Q}\hat{h}(r_{xxxx}+\Phi)\dxt
	+\EE\int_{Q}r(-\hat{h}_{xxxx}
	+\lambda^7\mu^8\xi^7\theta^2r+\chi_{G_0}\hat{v})\dxt\\
	&+\EE\int_{Q}R(\lambda^4\mu^4\xi^5\theta^2R+\hat{V})\dxt =0.
\end{split}
\end{equation}
By boundary condition in (\ref{equationr}) and (\ref{equationh}), after integrating by parts, we obtain
\begin{equation}\label{40}
\begin{split}
	&\EE\int_{Q}\lambda^7\mu^8\xi^7\theta^2r^2\dxt
	+\EE\int_{Q}\lambda^4\mu^4\xi^5\theta^2R^2\dxt
	 \\
	=&-\EE\int_{Q}\hat{h}\phi_0\dxt
	+\EE\int_{Q}\hat{h}_{x}\phi_{1}\dxt
	-\EE\int_{Q}\hat{h}_{xx}\phi_{2}\dxt \\
	&-\EE\intt\int_{G_0}\hat{v}r\dxt
	-\EE\int_{Q}\hat{V}R\dxt.
\end{split}
\end{equation}
In view of (\ref{carest3}) in lemma \ref{carle3}, for any $\epsilon_{7}>0$, we use the Young inequality on the right-hand side of (\ref{40}) to obtain
\begin{equation}\label{41}
\begin{split}
	&\EE\int_{Q}\lambda^7\mu^8\xi^7\theta^2r^2\dxt
	+\EE\int_{Q}\lambda^4\mu^4\xi^5\theta^2R^2\dxt\\
	\leq & \epsilon_{7}\bigg(\EE \int_{Q}\theta^{-2}|\hat{h}|^2\dxt
	+\EE \int_{Q}\theta^{-2}\lambda^{-2}\mu^{-2}\xi^{-3}|\hat{h}_{x}|^2\dxt\\
	&+\EE \int_{Q}\theta^{-2}\lambda^{-4}\mu^{-4}\xi^{-5}|\hat{h}_{xx}|^2\dxt
	+\EE \intt\int_{G_0}\theta^{-2}\lambda^{-7}\mu^{-8}\xi^{-7}|\hat{v}|^2\dxt\\
	&+\EE \int_{Q}\theta^{-2}\lambda^{-4}\mu^{-4}\xi^{-5}|\hat{V}|^2\dxt \bigg)
	\\
	&+C_{\epsilon_{7}}\bigg(\EE \int_{Q}\theta^{2}|\phi_{0}|^2\dxt
	+\EE \int_{Q}\theta^{2}\lambda^{2}\mu^{2}\xi^{3}|\phi_{1}|^2\dxt\\
	&+\EE \int_{Q}\theta^{2}\lambda^{4}\mu^{4}\xi^{5}|\phi_{2}|^2\dxt
	+\EE \intt\int_{G_0}\theta^{2}\lambda^{7}\mu^{8}\xi^{7}|r|^2\dxt\\
	&+\EE \int_{Q}\theta^{2}\lambda^{4}\mu^{4}\xi^{5}|R|^2\dxt\bigg).
\end{split}
\end{equation}
By choosing $\epsilon_{7}>0$ small enough in (\ref{41}) and using the Carleman estimate (\ref{carest3}) with the initial state $h_{0}=0$, we obtain that
\begin{equation}\label{42}
\begin{split}
	&\EE\int_{Q}\lambda^7\mu^8\xi^7\theta^2r^2\dxt \\
	\leq
	&C\bigg(\EE \int_{Q}\theta^{2}|\phi_{0}|^2\dxt
	+\EE \int_{Q}\theta^{2}\lambda^{2}\mu^{2}\xi^{3}|\phi_{1}|^2\dxt\\
	&+\EE \int_{Q}\theta^{2}\lambda^{4}\mu^{4}\xi^{5}|\phi_{2}|^2\dxt
	+\EE \intt\int_{G_0}\theta^{2}\lambda^{7}\mu^{8}\xi^{7}|r|^2\dxt\\
	&+\EE \int_{Q}\theta^{2}\lambda^{4}\mu^{4}\xi^{5}|R|^2\dxt\bigg).
\end{split}
\end{equation}

Next, let us estimate the terms
$r_{x}$ and $r_{xx}$, which will be done by a weighted energy estimate for (\ref{equationr}). In fact , by using It\^o's formula to calculate $\dd(\lambda^3\mu^4\xi^3\theta^2r^2) $, we have
\begin{equation}\label{66}
\begin{split}
&\EE\int_{G}
\lambda^3\mu^4\xi^3(0)\theta^2(0)r^2(0)\dx
+2\EE\int_{Q}
\lambda^3\mu^4\xi^3\theta^2|r_{xx}|^2\dxt
\\
=\ & -\EE\int_{Q}
\lambda^3\mu^4(\xi^3\theta^2)_{t}r^2\dxt
-4\EE\int_{Q}
\lambda^3\mu^4(\xi^3\theta^2)_{xx}r_{xx}r\dxt
\\
&+\EE\int_{Q}
\lambda^3\mu^4(\xi^3\theta^2)_{xxxx}r^2\dxt
-2\EE\int_{Q}
\lambda^3\mu^4\xi^3\theta^2r\Phi\dxt
\\
&-\EE\int_{Q}
\lambda^3\mu^4\xi^3\theta^2R^2\dxt:=\sum_{i=1}^{5}\mathcal{W}_{i}.
\end{split}	
\end{equation}
Noting that
\begin{equation}\nonumber
-(\xi^{3}\theta^{2})_{t}
=-(3+2\lambda\varphi)\xi^{3}\theta^{2}\frac{\gamma_{t}}{\gamma}
\leq -C\lambda \xi^{3}|\varphi||\gamma_{t}|\theta^{2}
\leq 0
\end{equation}
on $(t,x)\in [0,T/4]\times G$ and
\begin{equation}\nonumber
\begin{split}
|(\xi^{3}\theta^{2})_{t}|
\leq  C(\lambda+1)\xi^{3}\theta^{2}|\varphi\gamma|
\leq C(\lambda\mu+\mu)\xi^{5}\theta^{2}
\leq C\lambda\mu\xi^{5}\theta^{2}
\end{split}
\end{equation}
on $(t,x)\in [T/4,T]\times G$, we have
\begin{equation}\nonumber
\begin{split}
\mathcal{W}_{1}
\leq\ &
-\EE \int_{0}^{T/4}\int_{G}\lambda^{3}\mu^{4}(\xi^{3}\theta^{2})_{t}r^2\dxt
+\EE \int_{T/4}^{T}\int_{G}\lambda^{3}\mu^{4}|(\xi^{3}\theta^{2})_{t}|r^2\dxt
\\
\leq\ &\EE \int_{T/4}^{T}\int_{G}\lambda^{4}\mu^{5}\xi^{5}\theta^{2}r^2\dxt.
\end{split}	
\end{equation}
Using Young inequality, by direct computation, we obtain that
\begin{equation}\nonumber
\begin{split}
|\mathcal{W}_{2}|
\leq \ &
\epsilon_{8} \EE \int_{Q}\lambda^{3}\mu^{4}\xi^{3}\theta^{2}|r_{xx}|^2\dxt
+C_{\epsilon_{8}}\EE \int_{Q}\lambda^{7}\mu^{8}\xi^{7}\theta^{2}|r|^2\dxt,
\\
|\mathcal{W}_{3}|
\leq \ &
C\EE \int_{Q}\lambda^{7}\mu^{8}\xi^{7}\theta^{2}|r|^2\dxt,
\\
|\mathcal{W}_{4}|
\leq \ &
C\bigg(\EE \int_{Q}\theta^2\lambda^{7}\mu^{8}\xi^{7}r^2\dxt
+\EE \int_{Q}\theta^2\lambda^{5}\mu^{6}\xi^{5}|r_{x}|^2\dxt
\bigg)
+\epsilon_{9}\EE \int_{Q}\theta^2\lambda^{3}\mu^{4}\xi^{3}|r_{xx}|^2\dxt
\\
&+C\bigg(\EE \int_{Q}\theta^2|\phi_{0}|^2\dxt
+\EE \int_{Q}\theta^2\lambda\mu^{2}\xi|\phi_{1}|^2\dxt
\bigg)
+C_{\epsilon_{9}}\EE \int_{Q}\theta^2\lambda^{3}\mu^{4}\xi^{3}|\phi_{2}|^2\dxt.
\end{split}
\end{equation}
By choosing $\epsilon_{8}>0$ and $\epsilon_{9}>0$ small enough, combining the above inequalities with (\ref{66}), we arrive at the following estimate
\begin{equation}\label{67}
\begin{split}
&\EE\int_{G}
\lambda^3\mu^4\xi^3(0)\theta^2(0)r^2(0)\dx
+\EE\int_{Q}
\lambda^3\mu^4\xi^3\theta^2|r_{xx}|^2\dxt
\\
\leq \ &
C\bigg(\EE \int_{Q}\theta^2\lambda^{7}\mu^{8}\xi^{7}r^2\dxt
+\EE \int_{Q}\theta^2\lambda^{5}\mu^{6}\xi^{5}|r_{x}|^2\dxt
\bigg)
\\
&+C\bigg(\EE \int_{Q}\theta^2|\phi_{0}|^2\dxt
+\EE \int_{Q}\theta^2\lambda\mu^{2}\xi|\phi_{1}|^2\dxt
+\EE \int_{Q}\theta^2\lambda^{3}\mu^{4}\xi^{3}|\phi_{2}|^2\dxt
\bigg).
\end{split}	
\end{equation}

In addition, by  integration by parts and boundary conditions in (\ref{equationr}), we have
\begin{equation}\nonumber
\begin{split}
\EE \int_{Q}\theta^2\lambda^{5}\mu^{6}\xi^{5}|r_{x}|^2\dxt
=-\EE \int_{Q}\theta^2\lambda^{5}\mu^{6}\xi^{5}r_{xx}r\dxt
+\frac{1}{2}\EE \int_{Q}(\theta^2\lambda^{5}\mu^{6}\xi^{5})_{xx}r^2\dxt.
\end{split}	
\end{equation}
Further, by Young inequality, we obtain that
\begin{equation}\label{68}
\begin{split}
&\EE \int_{Q}\theta^2\lambda^{5}\mu^{6}\xi^{5}|r_{x}|^2\dxt
\\
\leq\
&\epsilon_{10} \EE \int_{Q}\theta^2\lambda^{3}\mu^{4}\xi^{3}|r_{xx}|^2\dxt
+C_{\epsilon_{10}} \EE \int_{Q}\theta^2\lambda^{7}\mu^{8}\xi^{7}|r|^2\dxt.
\end{split}	
\end{equation}
By choosing $\epsilon_{10}>0$ small enough, combining (\ref{67}) and (\ref{68}), we have
\begin{equation}\label{69}
\begin{split}
&\EE\int_{G}
\lambda^3\mu^4\xi^3(0)\theta^2(0)r^2(0)\dx
+\EE\int_{Q}
\lambda^3\mu^4\xi^3\theta^2|r_{xx}|^2\dxt
\\
\leq \ &
C\bigg(\EE \int_{Q}\theta^2\lambda^{7}\mu^{8}\xi^{7}r^2\dxt
+\EE \int_{Q}\theta^2|\phi_{0}|^2\dxt
\\
&+\EE \int_{Q}\theta^2\lambda\mu^{2}\xi|\phi_{1}|^2\dxt
+\EE \int_{Q}\theta^2\lambda^{3}\mu^{4}\xi^{3}|\phi_{2}|^2\dxt
\bigg).
\end{split}	
\end{equation}

Finally, from (\ref{42}), (\ref{68}) and (\ref{69}), we obtain (\ref{carest1}). Thus, the proof of Theorem \ref{carle1} is complete.
\end{proof}

\section{The proof of Theorem \ref{control}}

In this section, we will give the proof of the Theorem \ref{control}. Applying Theorem \ref{carle1}, we first  establish a controllability result for a linear stochastic system as mentioned in section 1, that is,
\begin{equation}\label{equationy3}
	\left\{
		\begin{aligned}
	&\dd y+y_{xxxx}\dt=(\phi+\chi_{G_0}u)\dt+U\dd B(t) &\textup{in}\ &Q,\\
    &y(t,0)=y(t,1)=0 &\textup{in}\ &(0,T),\\
    &y_{x}(t,0)=y_{x}(t,1)=0 &\textup{in}\ &(0,T),\\
    &y(0)=y_0 &\textup{in}\ &G,
            \end{aligned}
    \right.
\end{equation}
where $y=y(t,x)$ denotes the state variable associated to the initial state $y_{0}\in L^2_{\mathcal{F}_{0}}(\Omega;L^2(G))$, the control pair $(u,U)\in L^2_{\mathbb{F}}(0,T;L^2(G_{0}))\times L^2_{\mathbb{F}}(0,T;L^2(G))$.
Observe that given the aforementioned regularity on the controls and source term, one can easily show that system $(\ref{equationy3})$ admits a unique solution $y\in \mathcal{H}_{T}$ (see e.g., \cite[Theorem 3.24]{Lu2021Mathematical}).

We define the space
\begin{equation}\nonumber
	{\mathscr{S}_{\lambda,\mu}}=\bigg\{\phi\in L^2_{\mathbb{F}}(0,T;L^2(G))\ \bigg|\ \bigg(\EE\int_{Q}\theta^{-2}\lambda^{-7}\mu^{-8}\xi^{-7}|\phi|^2\dxt\bigg)^{\frac{1}{2}}<\infty \bigg\},
\end{equation}
which is a Banach space equipped with the canonical norm denoted by $\|\cdot\|_{\mathscr{S}_{\lambda,\mu}}$.

\begin{theorem}\label{control1}
Assume that $\phi\in L^2_{\mathbb{F}}(0,T;L^2(G))$. For any $y_0\in L^2_{\mathcal{F}_{0}}(\Omega;L^2(G))$, there is a control pair $(\hat{u},\hat{U})$ such that the associated solution $\hat{y}$ to the controlled  system $(\ref{equationy3})$ satisfies $\hat y(T)=0$ in $G$, $\mathbb{P}$-a.s. Moreover, one can find two positive constants $\lambda_{0}$ and $\mu_{0}$, such that
\begin{equation}\label{control1est}
\begin{split}
&\EE\intt\int_{G_0}\theta^{-2}\lambda^{-7}\mu^{-8}\xi^{-7}|\hat u|^2\dxt
+\EE\int_{Q}\theta^{-2}\lambda^{-4}\mu^{-4}\xi^{-5}|\hat U|^2\dxt
\\
&+\EE\int_{Q}\theta^{-2}|\hat y|^2\dxt
+\EE\int_{Q}\theta^{-2}\lambda^{-2}\mu^{-2}\xi^{-3}|\hat y_{x}|^2\dxt
\\
&+\EE\int_{Q}\theta^{-2}\lambda^{-4}\mu^{-4}\xi^{-5}|\hat y_{xx}|^2\dxt
\\
\leq \
	&C\bigg(\EE\int_{G}\lambda^{-3}\mu^{-4}e^{-30\mu m}\theta^{-2}(0)|y_{0}|^2\dx
	+\|\phi\|^2_{\mathscr{S}_{\lambda,\mu}}\bigg),
\end{split}
\end{equation}
for all $\lambda\geq\lambda_{0}$ and $\mu\geq\mu_{0}$, where  $C>0$ only depends on $G$ and $G_0$.
\end{theorem}

\begin{proof}[\bf Proof of  Theorem \ref{control1}:]
Similar to the proof of lemma \ref{carle3}, we also divide the proof into three parts.

{\it Step 1.} Using $\theta_{\varepsilon}$ defined in lemma \ref{carle3}, we introduce the functional
\begin{equation}\nonumber
\begin{split}
	J_{\varepsilon}(u,U)
	:=&\frac{1}{2}\EE \int_{Q}\theta_{\varepsilon}^{-2}|y|^2\dxt
	+\frac{1}{2}\EE \int_{Q}\theta_{\varepsilon}^{-2}\lambda^{-2}\mu^{-2}\xi^{-3}|y_{x}|^2\dxt\\
	&+\frac{1}{2}\EE \int_{Q}\theta_{\varepsilon}^{-2}\lambda^{-4}\mu^{-4}\xi^{-5}|y_{xx}|^2\dxt
	+\frac{1}{2}\EE \intt\int_{G_0}\theta^{-2}\lambda^{-7}\mu^{-8}\xi^{-7}|u|^2\dxt
	\\
	&+\frac{1}{2}\EE \int_{Q}\theta^{-2}\lambda^{-4}\mu^{-4}\xi^{-5}|U|^2\dxt
	+\frac{1}{2\varepsilon}\EE \int_{G}|y(T)|^2\dx
\end{split}	
\end{equation}
and consider the following optimal control problem:
\begin{equation}\label{70}
\begin{split}
\left\{
\begin{aligned}
	&\min_{(u,U)\in \mathscr{U}}J_{\varepsilon}(u,U)\\
	&\text{subject to equation}\ (\ref{equationy3}),
\end{aligned}
\right.
\end{split}
\end{equation}
where
\begin{equation}\nonumber
\begin{split}
	\mathscr{U}=\bigg\{&(u,U)\in L^2_\mathbb{F}(0,T;L^2(G_0))\times L^2_\mathbb{F}(0,T;L^2(G))\ \bigg|\\
	&\EE \intt\int_{G_0}\theta^{-2}\lambda^{-7}\mu^{-8}\xi^{-7}|u|^2\dxt<\infty,\
	\EE \int_{Q}\theta^{-2}\lambda^{-6}\mu^{-6}\xi^{-6}|U|^2\dxt<+\infty \bigg\}.
\end{split}	
\end{equation}
Similar to  \cite{Li1995Optimal}, it is easy to check that for any $\varepsilon>0$, (\ref{70}) admits a unique optimal pair solution that we denote by $(u_{\varepsilon},U_{\varepsilon})$. Moreover, by the standard variational method (see \cite{Li1995Optimal,Lion1971Optimal}), it follows that
\begin{equation}\label{50}
	u_{\varepsilon}=-\chi _{G_0}\theta^{2}\lambda^{7}\mu^{8}\xi^{7}z_{\varepsilon},\quad
	U_{\varepsilon}=-\theta^{2}\lambda^{4}\mu^{4}\xi^{5}Z_{\varepsilon}\quad \text{in}\ Q,\ \mathbb{P}\text{-}a.s.,
\end{equation}
where the pair $(z_{\varepsilon},Z_{\varepsilon})$ verifies the backward stochastic equation
\begin{equation}\label{equationz1}
	\left\{
		\begin{aligned}
	&\dd z_{\varepsilon}-z_{\varepsilon xxxx}\dt
	=-\Xi\dt
	+Z_{\varepsilon}\dd B(t) &\textup{in}\ &Q,\\
    &z_{\varepsilon}(t,0)=z_{\varepsilon}(t,1)=0 &\textup{in}\ &(0,T),\\
    &z_{\varepsilon x}(t,0)=z_{\varepsilon x}(t,1)=0 &\textup{in}\ &(0,T),\\
    &z_{\varepsilon}(T)=\frac{1}{\varepsilon}y_{\varepsilon}(T) &\textup{in}\ &G,
            \end{aligned}
    \right.
\end{equation}
where
$y_{\varepsilon}$ is the solution of
(\ref{equationy3}) with the controls $(u,U)=(u_{\varepsilon},U_{\varepsilon})$
and
\begin{equation}\nonumber
 \Xi=\theta_{\varepsilon}^{-2}y_{\varepsilon}-(\theta_{\varepsilon}^{-2}\lambda^{-2}\mu^{-2}\xi^{-3}y_{\varepsilon x})_x+(\theta_{\varepsilon}^{-2}\lambda^{-4}\mu^{-4}\xi^{-5}y_{\varepsilon xx})_{xx}.
\end{equation}

{\it Step 2.} We now establish a uniform estimate for the optimal solutions $\{(y_{\varepsilon},u_{\varepsilon},U_{\varepsilon})\}_{\varepsilon>0}$.
By It\^o's formula, (\ref{equationy3}) and (\ref{equationz1}), it follows that
\begin{equation}\nonumber
\begin{split}
	&\EE\int_{G}y_{\varepsilon}(T){z_{\varepsilon}}(T)\dx \\
	=& \EE\int_{G}y_{\varepsilon}(0){z_{\varepsilon}}(0)\dx
	+\EE\int_{Q}[-y_{\varepsilon xxxx}+\phi+\chi_{G_0}u_{\varepsilon}]{z_{\varepsilon}}\dxt \\
	&+\EE\int_{Q}[z_{\varepsilon xxxx}
	-\Xi]y_{\varepsilon}\dxt
	+\EE\int_{Q}U_{\varepsilon}{Z_{\varepsilon}}\dxt .
\end{split}
\end{equation}
This, together with (\ref{50}) and the last equality of (\ref{equationz1}), imply that
\begin{equation}\label{51}
\begin{split}
&\EE\intt\int_{G_0}\theta^{2}\lambda^{7}\mu^{8}\xi^{7}|z_{\varepsilon}|^2\dxt
+\EE\int_{Q}\theta^{2}\lambda^{4}\mu^{4}\xi^{5}|Z_{\varepsilon}|^2\dxt
    +\EE\int_{Q}\theta_{\varepsilon}^{-2}|y_{\varepsilon}|^2\dxt\\
    &+\EE\int_{Q}\theta_{\varepsilon}^{-2}\lambda^{-2}\mu^{-2}\xi^{-3}|y_{\varepsilon x}|^2\dxt
    +\EE\int_{Q}\theta_{\varepsilon}^{-2}\lambda^{-4}\mu^{-4}\xi^{-5}|y_{\varepsilon xx}|^2\dxt\\
	&+\frac{1}{\varepsilon}\EE\int_{G}|y_{\varepsilon}(T)|^2\dx
	= \EE\int_{G}y_{0}{z_{\varepsilon}}(0)\dx
	+\EE\int_{Q}\phi{z_{\varepsilon}}\dxt  .
\end{split}
\end{equation}
Noting that $\Xi\in L^{2}_{\mathbb{F}}(0,T;H^{-2}(G))$, applying the Carleman estimate (\ref{carest1}) in Theorem \ref{carle1} to (\ref{equationz1}) , we get
\begin{equation}\label{52}
    \begin{split}
	&\EE\int_{G}\lambda^3\mu^4e^{30\mu m}\theta^2(0)|z_{\varepsilon}(0)|^2\dx
	+\EE\int_{Q}\lambda^3\mu^4\xi^3\theta^2|z_{\varepsilon xx}|^2\dxt
	\\
	&+\EE \int_{Q}\theta^2\lambda^{5}\mu^{6}\xi^{5}|z_{\varepsilon x}|^2\dxt
	+\EE\int_{Q}\lambda^7\mu^8\xi^7\theta^2|z_{\varepsilon}|^2\dxt
	\\
	\leq
	&C\bigg(\EE \int_{Q}\theta_{\varepsilon}^{-2}|y_{\varepsilon}|^2\dxt
	+\EE \int_{Q}\theta_{\varepsilon}^{-2}\lambda^{-2}\mu^{-2}\xi^{-3}|y_{\varepsilon x}|^2\dxt\\
	&+\EE \int_{Q}\theta_{\varepsilon}^{-2}\lambda^{-4}\mu^{-4}\xi^{-5}|y_{\varepsilon xx}|^2\dxt
	+\EE \intt\int_{G_0}\theta^{2}\lambda^{7}\mu^{8}\xi^{7}|z_{\varepsilon}|^2\dxt\\
	&+\EE \int_{Q}\theta^{2}\lambda^{4}\mu^{4}\xi^{5}|Z_{\varepsilon}|^2\dxt\bigg),
	\end{split}
	\end{equation}
where we have used $\theta\theta_{\varepsilon}^{-1}\leq 1$ for any $(t,x)\in Q$.

In view of (\ref{52}), we use the Young inequality on the right-hand side of (\ref{51}) to obtain
\begin{equation}\label{53}
\begin{split}
&\EE\intt\int_{G_0}\theta^{2}\lambda^{7}\mu^{8}\xi^{7}|z_{\varepsilon}|^2\dxt
+\EE\int_{Q}\theta^{2}\lambda^{4}\mu^{4}\xi^{5}|Z_{\varepsilon}|^2\dxt
\\
&+\EE\int_{Q}\theta_{\varepsilon}^{-2}|y_{\varepsilon}|^2\dxt
+\EE\int_{Q}\theta_{\varepsilon}^{-2}\lambda^{-2}\mu^{-2}\xi^{-3}|y_{\varepsilon x}|^2\dxt
\\
&+\EE\int_{Q}\theta_{\varepsilon}^{-2}\lambda^{-4}\mu^{-4}\xi^{-5}|y_{\varepsilon xx}|^2\dxt
+\frac{1}{\varepsilon}\EE\int_{G}|y_{\varepsilon}(T)|^2\dx
\\
\leq \ &\epsilon_{11}\bigg( \EE\int_{G}\lambda^3\mu^4e^{30\mu m}\theta^2(0)|{z_{\varepsilon}}(0)|^2\dx
	+\EE\int_{Q}\theta^{2}\lambda^{7}\mu^{8}\xi^{7}|z_{\varepsilon}|^2\dxt\bigg)\\
	&+C_{\epsilon_{11}}\bigg(\EE\int_{G}\lambda^{-3}\mu^{-4}e^{-30\mu m}\theta^{-2}(0)|y_{0}|^2\dx
	+\EE\int_{Q}\theta^{-2}\lambda^{-7}\mu^{-8}\xi^{-7}|\phi|^2\dxt\bigg),
\end{split}
\end{equation}
for any $\epsilon_{11}>0$.
Using (\ref{52}) and (\ref{53}) with sufficiently small $\epsilon_{11}>0$, we obtain that
\begin{equation}\label{71}
\begin{split}
&\EE\intt\int_{G_0}\theta^{2}\lambda^{7}\mu^{8}\xi^{7}|z_{\varepsilon}|^2\dxt
+\EE\int_{Q}\theta^{2}\lambda^{4}\mu^{4}\xi^{5}|Z_{\varepsilon}|^2\dxt
\\
&+\EE\int_{Q}\theta_{\varepsilon}^{-2}|y_{\varepsilon}|^2\dxt
+\EE\int_{Q}\theta_{\varepsilon}^{-2}\lambda^{-2}\mu^{-2}\xi^{-3}|y_{\varepsilon x}|^2\dxt
\\
&+\EE\int_{Q}\theta_{\varepsilon}^{-2}\lambda^{-4}\mu^{-4}\xi^{-5}|y_{\varepsilon xx}|^2\dxt
+\frac{1}{\varepsilon}\EE\int_{G}|y_{\varepsilon}(T)|^2\dx
\\
\leq \
	&C\bigg(\EE\int_{G}\lambda^{-3}\mu^{-4}e^{-30\mu m}\theta^{-2}(0)|y_{0}|^2\dx
	+\EE\int_{Q}\theta^{-2}\lambda^{-7}\mu^{-8}\xi^{-7}|\phi|^2\dxt\bigg).
\end{split}
\end{equation}
Further, by applying (\ref{50}), (\ref{71}) imply that
\begin{equation}\label{54}
	\begin{split}
&\EE\intt\int_{G_0}\theta^{-2}\lambda^{-7}\mu^{-8}\xi^{-7}|u_{\varepsilon}|^2\dxt
+\EE\int_{Q}\theta^{-2}\lambda^{-4}\mu^{-4}\xi^{-5}|U_{\varepsilon}|^2\dxt
\\
&+\EE\int_{Q}\theta_{\varepsilon}^{-2}|y_{\varepsilon}|^2\dxt
+\EE\int_{Q}\theta_{\varepsilon}^{-2}\lambda^{-2}\mu^{-2}\xi^{-3}|y_{\varepsilon x}|^2\dxt
\\
\end{split}	
\end{equation}
\begin{equation}\nonumber
\begin{split}
&+\EE\int_{Q}\theta_{\varepsilon}^{-2}\lambda^{-4}\mu^{-4}\xi^{-5}|y_{\varepsilon xx}|^2\dxt
+\frac{1}{\varepsilon}\EE\int_{G}|y_{\varepsilon}(T)|^2\dx
\\
\leq \
	&C\bigg(\EE\int_{G}\lambda^{-3}\mu^{-4}e^{-30\mu m}\theta^{-2}(0)|y_{0}|^2\dx
	+\EE\int_{Q}\theta^{-2}\lambda^{-7}\mu^{-8}\xi^{-7}|\phi|^2\dxt\bigg).
\end{split}
\end{equation}

{\it Step 3.} By (\ref{54}),  it is easy to check that there exists $(\hat{u},\hat{U},\hat{y})$ such that
\begin{equation}\label{55}
		\begin{aligned}
	&u_{\varepsilon}\rightarrow \hat{u} &\textup{weakly in}\ &L^2(\Omega\times(0,T);L^2(G_0)),\\
    &U_{\varepsilon}\rightarrow  \hat{U}  &\textup{weakly in}\ &L^2(\Omega\times(0,T);L^2(G),\\
    &y_{\varepsilon}\rightarrow  \hat{y}  &\textup{weakly in}\ &L^2(\Omega\times(0,T);H_{0}^2(G)).
            \end{aligned}
\end{equation}

We claim that $\hat y$ is the solution to (\ref{equationy3}) associated to $(\hat u, \hat U)$. In fact, suppose that $\tilde y$ is the unique solution to (\ref{equationy3}) with controls $(\hat u, \hat U)$.
For any $\Lambda_{2}\in L^2_{\mathbb{F}}(0,T;H^{-2}(G))$, we consider $(\rho,\varrho)$  the unique solution to the backward equation
\begin{equation}\nonumber
	\left\{
		\begin{aligned}
	&\dd \rho-\rho_{xxxx}\dt
	=\Lambda_{2}\dt+\varrho\dd B(t) &\textup{in}\ &Q,\\
    &\rho(t,0)=\rho(t,1)=0 &\textup{in}\ &(0,T),\\
    &\rho_{x}(t,0)=\rho_{x}(t,1)=0 &\textup{in}\ &(0,T),\\
    &\rho (T)=0 &\textup{in}\ &G.
            \end{aligned}
    \right.
\end{equation}
Then, by applying It\^o's formula to calculate $\dd(y_{\varepsilon}\rho-\tilde y\rho)$, we obtain that
\begin{equation}\label{56}
\begin{split}
0=\EE\int_{Q}(y_{\varepsilon}-\tilde y)\Lambda_{2}\dxt
+\EE\intt\int_{G_{0}}(u_{\varepsilon}-\hat u)\rho \dxt
+\EE\int_{Q}(U_{\varepsilon}-\hat U)\varrho \dxt .
\end{split}	
\end{equation}
By taking the limit as $\varepsilon\rightarrow 0$ in (\ref{56}) and using (\ref{55}), we have
\begin{equation}\nonumber
\EE\int_{Q}(\hat y-\tilde y)\Lambda_{2}\dxt =0.
\end{equation}
Since $\Lambda_{2}\in L^2_{\mathbb{F}}(0,T;H^{-2}(G))$ is arbitrary, we deduce that $\hat y=\tilde y$ in $Q$, $\mathbb{P}$-a.s.

To conclude, we notice from (\ref{54}) that $\hat y(T)=0$ in $G$, $\mathbb{P}$-a.s. Moreover, from the weak convergence (\ref{55}), the uniform estimate (\ref{54}) and Fatou's lemma, we deduce (\ref{control1est}). Thus, the proof of Theorem \ref{control1} is complete.
\end{proof}
Based on the above results, let us prove the Theorem \ref{control}.
\begin{proof}[\bf Proof of Theorem \ref{control}:]
According to Theorem \ref{control1}, for any given $\phi\in L^2_{\mathbb{F}}(0,T;L^2(G))$, we know that there exists a control pair $(u,U)\in L^2_\mathbb{F}(0,T;L^2(G_0))\times L^2_\mathbb{F}(0,T;L^2(G))$, such that the corresponding solution $y\in \mathcal{H}_{T}$ to the controlled system (\ref{equationy3})  satisfies $y(T)=0$ in $G$, $\mathbb{P}$-a.s.
Hence, let us consider a nonlinearity $f$ satisfying assumptions $(A_{1})$, $(A_{2})$ and $(A_{3})$ and define the nonlinear map,
\begin{equation}\nonumber
	\mathscr{E}: \phi\in {\mathscr{S}_{\lambda,\mu}}\mapsto f(w,t,x,y,y_{x},y_{xx})\in {\mathscr{S}_{\lambda,\mu}},
\end{equation}
where $y$ is the trajectory of (\ref{equationy3}) associated to the data $y_{0}$ and $\phi$.
In the following, to simplify the notation, we simply write $f(y,y_{x},y_{xx})$.

Next, we will show that $\mathscr{E}$ is a contraction mapping from ${\mathscr{S}_{\lambda,\mu}}$ into ${\mathscr{S}_{\lambda,\mu}}$.
First, we check that the mapping $\mathscr{E}$ is well defined. In fact, for any $\phi\in {\mathscr{S}_{\lambda,\mu}}$, using $(A_{1})$-$(A_{3})$ and $(\ref{control1est})$, we have
\begin{equation}\nonumber
\begin{split}
\|\mathscr{E}\phi\|_{\mathscr{S}_{\lambda,\mu}}^2
=&\EE \int_{Q}\theta^{-2}\lambda^{-7}\mu^{-8}\xi^{-7}|f(y,y_{x},y_{xx})|^2\dxt \\
\leq &3\kappa^2 \EE \int_{Q}\theta^{-2}\lambda^{-7}\mu^{-8}\xi^{-7}(|y|^2+|y_{x}|^2+|y_{xx}|^2)\dxt \\
\leq
&3\kappa^2(\lambda^{-7}\mu^{-8}+\lambda^{-5}\mu^{-6}+\lambda^{-3}\mu^{-4})
\bigg(\EE\int_{Q}\theta^{-2}|y|^2\dxt
\\
&+\EE\int_{Q}\theta^{-2}\lambda^{-2}\mu^{-2}\xi^{-3}|y_{x}|^2\dxt
+\EE\int_{Q}\theta^{-2}\lambda^{-4}\mu^{-4}\xi^{-5}|y_{xx}|^2\dxt\bigg)
\\
\leq
&C\lambda^{-3}\mu^{-4}
\bigg(\EE\int_{G}\lambda^{-3}\mu^{-4}e^{-30\mu m}\theta^{-2}(0)|y_{0}|^2\dx
	+\|\phi\|^2_{\mathscr{S}_{\lambda,\mu}}\bigg)
\\
\leq &\lambda^{-3}\mu^{-4}\bigg(C_{1}\EE\|y_{0}\|^2_{L^2(G)}	+C\|\phi\|^2_{\mathscr{S}_{\lambda,\mu}}\bigg)
< \infty,
\end{split}	
\end{equation}
for any sufficiently large parameters $\lambda\geq\lambda_{0}$ and $\mu\geq\mu_{0}$, where $C_{1}>0$ depends on $G$,  $G_{0}$, $\lambda$ and $\mu$, and $C > 0$ only depends on $G$ and $G_{0}$.
This proves that $\mathscr{E}$ is well defined.

Next, we check that the mapping $\mathscr{E}$ is strictly contractive.
Let $y_{1}$, $y_{2}$ be the solutions of the controlled system (\ref{equationy3}) with respect to the source terms $\phi^{(1)}$, $\phi^{(2)}\in {\mathscr{S}_{\lambda,\mu}}$, respectively.
Then applying (\ref{control1est}) in Theorem \ref{control1} for the equation associated to  $\phi=\phi^{(1)}-\phi^{(2)}$, $y(0)=y_{0}-y_{0}=0$, and using assumption $(A_{3})$, we have
\begin{equation}\nonumber
\begin{split}
	\|\mathscr{E}\phi^{(1)}-\mathscr{E}\phi^{(2)}\|_{\mathscr{S}_{\lambda,\mu}}
	=&\EE \int_{Q}\theta^{-2}\lambda^{-7}\mu^{-8}\xi^{-7}|f(y_{1},y_{1x},y_{1xx})-f(y_{2},y_{2x},y_{2xx})|^2\dxt \\
	\leq &3\kappa^2 \EE \int_{Q}\theta^{-2}\lambda^{-7}\mu^{-8}\xi^{-7}(|y_{1}-y_{2}|^2+|y_{1x}-y_{2x}|^2+|y_{1xx}-y_{2xx}|^2)\dxt \\
	 \leq &C\kappa^2\lambda^{-3}\mu^{-4}\|\phi^{(1)}-\phi^{(2)}\|^2_{\mathscr{S}_{\lambda,\mu}},
\end{split}	
\end{equation}
for any sufficiently large parameters $\lambda\geq\lambda_{0}$ and $\mu\geq\mu_{0}$, where $C > 0$ only depends on $G$ and $G_{0}$. Thus, if necessary, we can increase the value of $\lambda$ and $\mu$ such that $C\kappa^2\lambda^{-3}\mu^{-4}<1$, which implies that the mapping $\mathscr{E}$ is strictly contractive.

Further, by the Banach fixed point theorem, it follows that $\mathscr{E}$ has a unique fixed point $\tilde \phi\in {\mathscr{S}_{\lambda,\mu}}$. Moreover, it holds that $\tilde \phi=f(\omega,t,x,y,y_{x},y_{xx})$, where $y$ is the solution for the equation (\ref{equationy3}) associated to $\phi=\tilde \phi$, which means that $y$ is the solution to equation (\ref{equationy2}).
Applying Theorem \ref{control1}, we know that there exists a control pair $(u,U)\in L^2_\mathbb{F}(0,T;L^2(G_0))\times L^2_\mathbb{F}(0,T;L^2(G))$, such that $y$ to the controlled system (\ref{equationy2})  satisfies $y(T)=0$ in $G$, $\mathbb{P}$-a.s.
According to the Remark \ref{remark1},
the proof of Theorem \ref{control} is complete.
\end{proof}
%\section*{Acknowledgments}
%The authors sincerely thank the anonymous referees for their thorough reading and invaluable comments.

\section*{Appendix A}
In Appendix A, we give a brief proof for Remark \ref{remark01}.
\begin{proof}
We re-estimate of the boundary terms $\mathcal{J}_{4}$ in Theorem \ref{carle2} with the new boundary condition
\begin{equation}\nonumber
\left\{
		\begin{aligned}
    &r(t,0)=r(t,1)=0 &\textup{in}\ &(0,T),\\
    &r_{xx}(t,0)=r_{xx}(t,1)=0 &\textup{in}\ &(0,T).
         \end{aligned}
    \right.
\end{equation}
In this condition, we can easily check that
\begin{equation}\label{app1}\tag{A.1}
\begin{split}
&w(0,t)=w(1,t)=0,
\\
&w_{x}(0,t)=\theta r_{x}(0,t),\quad
w_{x}(1,t)=\theta r_{x}(1,t),
\\
&w_{xx}(0,t)=2\theta_{x} r_{x}(0,t)\quad
w_{xx}(1,t)=2\theta_{x} r_{x}(1,t),
\quad t\in (0,T).
\end{split}
\end{equation}

By the definition of ${\mathcal{B}}$, using (\ref{4}) and Young inequality, we obtain that
\begin{equation}\label{app2}\tag{A.2}
\begin{split}
\mathcal{J}_{4}
:=\ &\EE\int_{Q}\mathcal{B}\dx
\\
\geq\ &2\EE\int_{Q}(w_{xx}\dd w_{x})_{x}\dx
\\
&-\frac{3}{10}\EE\intt
 \lambda\mu\xi(1)\beta_{x}(1) |w_{xxx}(1)|^2\dt
 +\frac{3}{10}\EE\intt
 \lambda\mu\xi(0)\beta_{x}(0) |w_{xxx}(0)|^2\dt
 \\
 &-20\EE\intt \lambda^3\mu^3\xi^3(1)\beta_{x}^3(1)|w_{xx}(1)|^2 \dt
+20\EE\intt \lambda^3\mu^3\xi^3(0)\beta_{x}^3(0)|w_{xx}(0)|^2 \dt
\\
&-2\EE\intt \lambda^5\mu^5\xi^5(1)\beta_{x}^5(1)|w_{x}(1)|^2 \dt
+2\EE\intt \lambda^5\mu^5\xi^5(0)\beta_{x}^5(0)|w_{x}(0)|^2 \dt
\\
&-C\bigg[\EE\intt
 \mu |w_{xxx}(1)|^2\dt
+\EE\intt \lambda^2\mu^3\xi^2(1)|w_{xx}(1)|^2 \dt
\\
&+\EE\intt
 \mu |w_{xxx}(0)|^2\dt
+\EE\intt \lambda^2\mu^3\xi^2(0)|w_{xx}(0)|^2 \dt
\\
&+\EE\intt
 \lambda^4\mu^5\xi^4(0) |w_{x}(0)|^2\dt
+\EE\intt \lambda^4\mu^5\xi^4(0)|w_{x}(1)|^2 \dt
\bigg].
\end{split}	
\end{equation}
Further, by (\ref{app1}) and $w=\theta r$, we have
\begin{equation}\label{app3}\tag{A.3}
\begin{split}
w_{xx}\dd w_{x}
&=2\theta_{x}r_{x}\dd w_{x}
=2\theta_{x}\theta^{-1}w_{x}\dd w_{x}
=2\lambda\mu\xi\beta_{x}w_{x}\dd w_{x}
\\
&=\dd(\lambda\mu\xi\beta_{x}|w_{x}|^2)
-\lambda\mu\xi_{t}\beta_{x}|w_{x}|^2\dt
-\lambda\mu\xi\beta_{x}|\dd w_{x}|^2,
\quad (t,x)\in (0,T)\times \{0,1\}.
\end{split}
\end{equation}
Then, we have from (\ref{app3}) that
\begin{equation}\label{app4}\tag{A.4}
\begin{split}
&\EE\int_{Q}(w_{xx}\dd w_{x})_{x}\dx
\\
=\ &\EE\intt \dd(\lambda\mu\xi\beta_{x}|w_{x}|^2)\big|_{x=1}\dt
-\EE\intt \dd(\lambda\mu\xi\beta_{x}|w_{x}|^2)\big|_{x=0}\dt
\\
&-\EE\intt \lambda\mu\xi\beta_{x}|\dd w_{x}|^2\big|_{x=1}\dt
+\EE\intt \lambda\mu\xi\beta_{x}|\dd w_{x}|^2\big|_{x=0}\dt
\\
&-\EE\intt \lambda\mu\xi_{t}\beta_{x}|w_{x}|^2\big|_{x=1}\dt
+\EE\intt \lambda\mu\xi_{t}\beta_{x}|w_{x}|^2\big|_{x=0}\dt .
\end{split}	
\end{equation}
Using the fact that $\beta_{x}(1)<0$ and $\beta_{x}(0)>0$, we know that the first four terms on the right side of (\ref{app4}) are positive terms.
In addition, by the definition of $\gamma$ in (\ref{gamma}), we can get
\begin{equation}\label{app5}\tag{A.5}
\begin{split}
&\EE\int_{Q}(w_{xx}\dd w_{x})_{x}\dx
\geq -C\bigg[
\EE\intt
 \lambda^4\mu^5\xi^4(0) |w_{x}(0)|^2\dt
+\EE\intt \lambda^4\mu^5\xi^4(0)|w_{x}(1)|^2 \dt
\bigg].
\end{split}	
\end{equation}
Further, By virtue of (\ref{app2}), (\ref{app5}), $\beta_{x}(1)<0$, $\beta_{x}(0)>0$ and (\ref{beta}), we obtain that
\begin{equation}\nonumber
\begin{split}
\mathcal{J}_{4}
\geq\
&\frac{3}{10}\alpha_{0} \EE\intt
 \lambda\mu\xi(1) |w_{xxx}(1)|^2\dt
 +\frac{3}{10}\alpha_{0}\EE\intt
 \lambda\mu\xi(0) |w_{xxx}(0)|^2\dt
 \\
 &+20\alpha_{0}^3\EE\intt \lambda^3\mu^3\xi^3(1)|w_{xx}(1)|^2 \dt
+20\alpha_{0}^3\EE\intt \lambda^3\mu^3\xi^3(0)|w_{xx}(0)|^2 \dt
\\
&+2\alpha_{0}^5\EE\intt \lambda^5\mu^5\xi^5(1)|w_{x}(1)|^2 \dt
+2\alpha_{0}^5\EE\intt \lambda^5\mu^5\xi^5(0)|w_{x}(0)|^2 \dt
\\
&-C\bigg[\EE\intt
 \mu |w_{xxx}(1)|^2\dt
+\EE\intt \lambda^2\mu^3\xi^2(1)|w_{xx}(1)|^2 \dt
\\
&+\EE\intt
 \mu |w_{xxx}(0)|^2\dt
+\EE\intt \lambda^2\mu^3\xi^2(0)|w_{xx}(0)|^2 \dt
\\
&+\EE\intt
 \lambda^4\mu^5\xi^4(0) |w_{x}(0)|^2\dt
+\EE\intt \lambda^4\mu^5\xi^4(0)|w_{x}(1)|^2 \dt
\bigg].
\end{split}	
\end{equation}
The other part of the proof is similar to the proof of Theorem \ref{carle2}.
\end{proof}

\section*{Acknowledgments}

The research of HG was supported in part by National Key R\&D Program of China (No. 2023YFA1009002).
The research of GY was supported in part by NSFC (No. 11771074 and No.12371421) and National Key R\&D Program of China (No. 2021YFA1003400 and No. 2020YFA0714102).

\end{document}